\documentclass{amsart}
\usepackage{amsmath,amsthm,latexsym,amsfonts,amssymb,mathrsfs,array,manfnt}
\usepackage[all]{xy}
\usepackage{paralist}
\usepackage{framed}
\usepackage{cancel}
\usepackage{enumitem}
\usepackage{fullpage}
\usepackage{units}
\usepackage{gensymb}
\usepackage{setspace}
\usepackage{amscd}
\usepackage{tikz}
\usepackage{tikz-cd}
\usetikzlibrary{matrix}

\newcommand{\p}{\mathfrak{P}}
\newcommand{\into}{\hookrightarrow}

\newcommand{\abs}[1]{\left\lvert#1\right\rvert}
\newcommand{\norm}[1]{\left\lVert#1\right\rVert}

\newcommand{\Z}{\ensuremath{\mathbb{Z}}}
\newcommand{\R}{\ensuremath{\mathbb{R}}}
\newcommand{\C}{\ensuremath{\mathbb{C}}}
\newcommand{\Q}{\ensuremath{\mathbb{Q}}}

\newcommand{\M}{\mathcal{M}}

\newcommand{\isom}{\cong} 

\renewcommand{\P}{\ensuremath{\mathbb{P}}}
\renewcommand{\H}{\ensuremath{\mathcal{H}}}

\renewcommand{\bar}[1]{\overline{#1}}

\DeclareMathOperator{\Flags}{Flags}
\DeclareMathOperator{\fl}{fl}

\DeclareMathOperator{\Span}{Span}

\DeclareMathOperator{\rank}{Rank}

\DeclareMathOperator{\Rank}{Rank}

\DeclareMathOperator{\Sym}{Sym}

\DeclareMathOperator{\PS}{PS}

\usepackage{color}

\newcommand{\cH}{\mathcal{H}}

\newcommand{\cJ}{\mathcal{J}}

\newcommand{\cM}{\mathcal{M}}

\newcommand{\bA}{\mathbf{A}}
\newcommand{\bB}{\mathbf{B}}

\newcommand{\bP}{\mathbf{P}}

\theoremstyle{plain}
\newtheorem{theorem}{Theorem}
\numberwithin{theorem}{section}

\newtheorem{thm}[theorem]{Theorem}

\newtheorem{prop}[theorem]{Proposition}

\newtheorem{cor}[theorem]{Corollary}

\newtheorem{lem}[theorem]{Lemma}

\newtheorem{crit}[theorem]{Criteria}
\theoremstyle{definition}
\newtheorem{Definition/Theorem}[theorem]{Definition/Theorem}
\newtheorem{Definition/Proposition}[theorem]{Definition/Proposition}

\newtheorem{Def}[theorem]{Definition}
\newtheorem{Defthm}[theorem]{Definition-Theorem}
\newtheorem{ex}[theorem]{Example}

\newtheorem{Corollary/Definition}[theorem]{Corollary/Definition}

\theoremstyle{remark}

\newtheorem{rem}[theorem]{Remark}

\renewcommand{\H}{\cH}
\renewcommand{\M}{\cM}
\newcommand{\Hbar}{\bar{\cH}}
\newcommand{\Mbar}{\bar{\cM}}

\newcommand{\full}{\mathrm{full}}
\newcommand{\st}{\mathrm{st}}

\renewcommand{\rm}{\mathrm{rm}}

\renewcommand{\setminus}{\smallsetminus}
\renewcommand{\vert}{\mathrm{vert}}
\newcommand{\be}{\boldsymbol{\epsilon}}

\usetikzlibrary{arrows}
\usetikzlibrary{calc}
\newcommand{\rat}{
\tikz[minimum height=2ex]
  \path[dashed,->]
   node (a)            {}
   node (b) at (1em,0) {}
  ($(a.center)+(0,-.07)$) edge ($(b.center)+(0,-.07)$);
}

\newcommand{\multi}{
\tikz[minimum height=2ex]
  \path[->]
   node (a)            {}
   node (b) at (1em,0) {}
  ($(a.center)+(0,.0)$) edge ($(b.center)+(0,.0)$)
  ($(a.center)+(0,-.14)$) edge ($(b.center)+(0,-.14)$);
}

\DeclareMathOperator{\md}{md}
\DeclareMathOperator{\br}{br}

\DeclareMathOperator{\Th}{Th}
\DeclareMathOperator{\mrk}{Mark}
\DeclareMathOperator{\Verts}{Verts}
\DeclareMathOperator{\Edges}{Edges}

\DeclareMathOperator{\mpt}{mp}
\newcommand{\CP}{\C\P}
\newcommand{\dms}{\abs{\bP}-3}
\newcommand{\hvy}{hvy}
\newcommand{\lt}{lt}
\usepackage{hyperref}

\begin{document}

\title{Algebraic stability of meromorphic maps descended from Thurston's pullback maps}
\author{Rohini Ramadas}
\email{rohini\_ramadas@brown.edu}
\address{Department of Mathematics\\Brown University\\Providence, RI}
\thanks{This work was partially supported by NSF grants
0943832, 1045119, 1068190, and 1703308.}

\begin{abstract}
 Let $\phi:S^2 \to S^2$ be an orientation-preserving branched covering whose post-critical set has finite cardinality $n$. If $\phi$ has a fully ramified periodic point $p_{\infty}$ and satisfies certain additional conditions, then, by work of Koch, $\phi$ induces a meromorphic self-map $R_{\phi}$ on the moduli space $\M_{0,n}$; $R_{\phi}$ descends from Thurston's pullback map on Teichm\"uller space. Here, we relate the dynamics of $R_{\phi}$ on $\M_{0,n}$ to the dynamics of $\phi$ on $S^2$. Let $\ell$ be the length of the periodic cycle in which the fully ramified point $p_{\infty}$ lies; we show that $R_{\phi}$ is algebraically stable on the heavy-light Hassett space corresponding to $\ell$ heavy marked points and $(n-\ell)$ light points.   
\end{abstract}
\maketitle

\section{Introduction}\label{sec:Intro}

Suppose that $\phi:S^2 \to S^2$ is an orientation preserving branched covering from a topological $2$-sphere to itself, of topological degree $d>1$. A \emph{critical point} of $\phi$ is a point at which $\phi$ is not a local homeomorphism. If $x$ is a critical point of $\phi$ then $x$ has a punctured neighborhood on which $\phi$ is an  $r$-to$1$ covering map, with $2\le r\le d$. In this case the \emph{multiplicity} of $x$ is $(r-1)$; the map $\phi$ has $(2d-2)$ critical points counted with multiplicity. Suppose further that the \emph{post-critical set} of $\phi$:
$$\bP:=\{\phi^n(x)|\text{ $x$ is a critical point of $\phi$ and $n>0$}\}$$
is finite. Then $\phi$ is called \emph{post-critically finite/ PCF}. Thurston \cite{DouadyHubbard1993} introduced a holomorphic \emph{pullback} map $\Th_{\phi}$ induced by $\phi$ on the Teichm\"uller space $\mathcal{T}(S^2,\bP)$ of complex structures on $(S^2,\bP)$; the branched covering $\phi$ is homotopic to a PCF rational function on $\C\P^1$ if and only if $\Th_{\phi}$ has a fixed point. 

Teichm\"uller space $\mathcal{T}(S^2,\bP)$ is a non-algebraic complex manifold but is the universal cover of the algebraic moduli space $\M_{0,\bP}$ of markings of $\C\P^1$ by the set $\bP$. Koch has introduced algebraic dynamical systems on $\M_{0,\bP}$ that descend from the transcendental Thurston pullback map. We say that a critical point $x$ of $\phi$ is \emph{fully ramified} if it has the maximum possible multiplicity of $(d-1)$, i.e. if the local degree of $\phi$ at $x$ equals the global degree of $\phi$. We say that $y\in S^2$ is a \emph{periodic} point of $\phi$ if $\exists \ell>0$ such that $\phi^{\ell}(y)=y$; if $\ell$ is chosen to be minimal we say $y$ is periodic of period $\ell$. If $\ell=1$, i.e. if $\phi(y)=y$, then $y$ is a fixed point of $\phi$. We say $\phi$ is a \emph{topological polynomial} if there is a point on $S^2$ that is fully ramified and fixed. Now, suppose $\phi:(S^2,\bP)\to(S^2,\bP)$ is PCF and satisfies:
\begin{crit}\label{crit:KC}
\begin{enumerate}
\item $\bP$ contains a periodic and fully ramified point $p_\infty$ of $\phi$, and  \label{item:KC1}
\item either every other critical point of $\phi$ is also periodic or there is exactly one other critical point of $\phi$, \label{item:KC2}
\end{enumerate}
\end{crit}
then \cite{Koch2013} the ``inverse" of $\Th_{\phi}$ descends to $\M_{0,\bP}$. More precisely, there is a meromorphic map $R_{\phi}:\M_{0,\bP}\rat\M_{0,\bP}$ such that the following diagram commutes: 

\begin{center}
\begin{tikzcd}[row sep=18pt, column sep=80pt]
  \mathcal{T}(S^2,\mathbf{P})\arrow[r,"\textstyle\mathrm{Th}_\phi"]\arrow[d,"\begin{tabular}{c}universal cover\end{tabular}",swap]&\mathcal{T}(S^2,\mathbf{P})\arrow[d,"\begin{tabular}{c}universal cover\end{tabular}"]\\
  \mathcal{M}_{0,\mathbf{P}}&\mathcal{M}_{0,\mathbf{P}}\arrow[l,dashed,"\textstyle R_\phi"]
\end{tikzcd}
\end{center}

The moduli space $\M_{0,\bP}$ is not compact. It is natural to ask whether $R_{\phi}$ extends to a holomorphic self-map of some compactification. Projective space $\C\P^{\abs{\bP}-3}$ is a compactification of $\M_{0,\bP}$. Koch also showed that if the fully ramified point $p_\infty$ in criterion (\ref{crit:KC}, \ref{item:KC1}) is a fixed point of $\phi$, i.e. if $\phi$ is a topological polynomial, then $R_{\phi}:\C\P^{\abs{\bP}-3}\to\C\P^{\abs{\bP}-3}$ is holomorphic. Moreover, in this case, the union of the forward orbits of the critical loci of $R_{\phi}$ is an algebraic set in $\C\P^{\abs{\bP}-3}$, i.e. $R_{\phi}$ is a higher-dimensional analog of a post-critically finite map. 

In general, if $\phi$ is not a topological polynomial, we ask whether $R_{\phi}$ extends ``nicely" to some compactification of $\M_{0,\bP}$. It might be too much to expect that $R_{\phi}$ extends to a holomorphic self-map of some compactification. Instead, we study a weaker property called \emph{algebraic stability}. Let $X$ be some smooth projective compactification of $\M_{0,\bP}$, so $R_{\phi}$ can be considered to be a meromorphic self-map of $X$. Although $R_{\phi}:X\rat X$ may not extend to a holomorphic or even continuous map from $X$ to itself, it induces pullback actions $R_{\phi}^*$ on the singular cohomology groups of $X$ (see Section \ref{sec:Hurwitzcohomology} for details of how this action is defined). This action preserves the Hodge decomposition and therefore induces a pullback action on the groups $H^{k,k}(X)$. However, crucially, this action does not respect iteration, i.e. in general we do not have $(R_{\phi}^n)^*=(R_{\phi}^*)^n$. Suppose we do have, for some fixed $k$ and all $n>0$ that $(R_{\phi}^n)^*=(R_{\phi}^*)^n$ on $H^{k,k}(X)$; in this case we say that $R_{\phi}$ is \emph{$k$-stable} on $X$. We say $R_{\phi}$ is \emph{algebraically stable} on $X$ if it is $k$-stable on $X$ for all $k$. If $R_{\phi}$ extends to a holomorphic self-map of $X$ then it is automatically algebraically stable on $X$, so being algebraically stable may be thought of as `acting on cohomology like a holomorphic map does'.  

Koch and Roeder \cite{KochRoeder2015} showed that if $\phi$ has exactly two critical points, both periodic, then $R_{\phi}$ is algebraically stable on the Deligne-Mumford compactification of $\M_{0,\bP}$. This was generalized by Koch, Speyer and the author \cite{Ramadas2015}: If $\phi$ is PCF and $R_{\phi}$ exists, then $R_{\phi}$ is algebraically stable on the Deligne-Mumford compactification. The Deligne-Mumford compactification $\Mbar_{0,\bP}$ of $\M_{0,\bP}$ is ``large" as measured by the ranks of its cohomology groups and the number of irreducible components of $\Mbar_{0,\bP}\setminus\M_{0,\bP}$. On the other hand, by \cite{Koch2013}, if $\phi$ satisfies criteria (\ref{crit:KC}, \ref{item:KC1}) and (\ref{crit:KC}, \ref{item:KC2}) and is a topological polynomial, then $R_{\phi}$ is holomorphic, thus algebraically stable, on the much smaller compactification $\C\P^{\abs{\bP}-3}$. 

In this paper, we interpolate between \cite{Koch2013} and \cite{Ramadas2015} by identifying a relationship between the topological dynamics of $\phi$ and the algebraic dynamics of $R_{\phi}$. We find a sequence $\{X_{\ell}\}_{\ell=1,\ldots,\abs{\bP}}$ of of smooth projective compactifications of $\M_{0,\bP}$, with $X_1=\C\P^{\abs{\bP}-3}$ and $X_{\abs{\bP}}=\Mbar_{0,\bP}$, such that for all $\ell$, there is a birational holomorphic map $\rho_{\ell+1,\ell}: X_{\ell+1}\to X_{\ell}$. We show:

\begin{thm}\label{thm:Main}
If $\phi$ is a branched covering with post-critical set $\bP$ satisfying criteria (\ref{crit:KC}, \ref{item:KC1}) and (\ref{crit:KC}, \ref{item:KC2}) and such that the fully ramified point $p_\infty$ of (\ref{crit:KC}, \ref{item:KC1}) is in a cycle of period $\ell$, then the meromorphic map $R_{\phi}:\M_{0,\bP}\rat\M_{0,\bP}$ is algebraically stable on $X_{\ell}$.
\end{thm}

The $\ell$-th compactification $X_{\ell}$ is the \emph{heavy/light Hassett space} corresponding to $\ell$ heavy weights and $(\abs{\bP}-\ell)$ light weights, constructed by Hassett and parametrizing \emph{weighted stable curves} (\cite{Hassett2003}, see Sections \ref{sec:CombinatorialCompactifications} and \ref{sec:WeightedStableCurves} for details). The space $X_\ell$ can be obtained as an iterated blow-up of $\C\P^{\abs{P}-3}$. The last three compactifications, $X_{\abs{\bP}-2}$, $X_{\abs{\bP}-1}$ and $X_{\abs{\bP}}$, are isomorphic to each other, but for $\ell\le(\abs{P}-3)$, the birational map $\rho_{\ell+1,\ell}: X_{\ell+1}\to X_{\ell}$ contracts in dimension certain subvarieties in the boundary $X_{\ell+1}\setminus\M_{0,\bP}$. Under the pushforward map $(\rho_{\ell+1,\ell})_*$ on homology, the classes of the contracted subvarieties go to zero. Thus for $\ell=1,\ldots,(\abs{P}-2)$, the spaces $X_{\ell}$ are all distinct. For small $\ell$, the compactification $X_{\ell}$ is ``small", as measured by the number of components in its boundary $X_{\ell}\setminus\M_{0,\bP}$ and the ranks of its cohomology groups. If $\phi$ is a branched covering with a fully ramified point $p_\infty$ in a periodic cycle (i.e. satisfying criterion (\ref{crit:KC}, \ref{item:KC1})), then the length $\ell$ of that cycle measures how much $\phi$ ``resembles" a topological polynomial: If $\ell=1$ then $\phi$ is a topological polynomial; if $\ell>1$ is small then $\phi$ resembles a topological polynomial. We give a non-rigorous interpretation for Theorem \ref{thm:Main}:

\medskip

\noindent\textit{If $\phi$ resembles a topological polynomial, then $R_{\phi}$ is algebraically stable on a small compactification of $\M_{0,\bP}$.}

\subsection{Dynamical degrees and the significance of algebraic stability.} Let $g:U\rat U$ be a meromorphic self-map of a smooth quasiprojective complex variety, and let $X$ be some smooth projective compactification of $U$. As discussed above and described in Section \ref{sec:Hurwitzcohomology}, $g$ induces a pullback action on $H^{k,k}(X)$, but we may not have $(g^n)^*=(g^*)^n$. However, we obtain an important numerical invariant of $g$ by considering the asymptotics of the operators $(g^n)^*$. Pick any norm on $H^{k,k}(X)$. The \emph{$k$-th dynamical degree} of $g$ is the non-negative real number 
\begin{align*}
  \lim_{n\to\infty}\norm{(g^n)^*:H^{k,k}(X)\to H^{k,k}(X)}^{1/n}.
\end{align*}
This limit exists, is independent of the choice of norm, and also of the choice of compactification $X$ (Dinh and Sibony \cite{DinhSibony2005} in the complex setting and Truong \cite{Truong2015} in the algebraic setting). Thus the $k$-th dynamical degree is intrinsic to the action of $g$ on the possibly non-compact space $U$. The dynamical degrees of a map measure its complexity: The topological entropy of a holomorphic map is equal to the logarithm of its largest dynamical degree (Gromov \cite{Gromov2003} and Yomdin \cite{Yomdin1987}) and the topological entropy of a meromorphic map is at most the logarithm of its largest dynamical degree (Dinh and Sibony \cite{DinhSibony2005}).

Given $g:U\rat U$, if there exists a compactification $X$ of $U$ on which $g$ is $k$-stable, then the $k$-th dynamical degree of $g$ is the absolute value of the largest eigenvalue of $g^*$ acting on $H^{k,k}(X)$, thus an algebraic integer whose degree over $\Q$ is at most $\Rank(H^{k,k}(X))$. The degree over $\Q$ of an algebraic integer is a measure of its complexity. Thus if $g$ is $k$-stable on $X$, then $\Rank(H^{k,k}(X))$ gives an upper bound on a certain type of complexity of the map $g$.

A common strategy to compute the dynamical degrees of a given map is to look for birational models on which the map is $k$-stable/algebraically stable. However, Favre \cite{Favre2003} has given examples of monomial maps $g:\P^2\rat \P^2$ for which no such birational models exist. Computing the $k$-th dynamical degree of a map which is either provably not $k$-stable or not known to be $k$-stable on any model involves dealing with the pullbacks along infinitely many iterates, and is difficult to impossible. Also, given a meromorphic map, there is no known strategy to find a birational model on which it is $k$-stable. Thus there are only a few examples of meromorphic maps whose dynamical degrees have been computed. 

In this regard, monomial maps are perhaps the best understood. A monomial map $g:(\C^*)^n\to(\C^*)^n$ is determined by an $n$-by-$n$ integer matrix $M_g$ of exponents. Work of Jonsson and Wulcan \cite{JonssonWulcan2011} for $k=1$ and Lin \cite{LinStability} in general gives criteria on $M_g$ for the existence of a compactification on which $g$ is $k$-stable. When those criteria are satisfied, they give explicit descriptions of toric compactifications on which the maps are $1$-stable (Jonsson-Wulcan)/algebraically stable (Lin). These works also lead to formulas for the dynamical degrees of monomial maps: the $k$-th dynamical degree of $g$ is the absolute value of the product of the $k$ largest eigenvalues of the integer matrix $M_g$, thus an algebraic integer of degree at most $\binom{n}{k}$ \cite{Lin2012}. In addition to monomial maps, birational surface transformations are also well-studied: Diller and Favre (\cite{DillerFavre2001}) showed that every birational transformation $g$ of a projective surface $X$ is $1$-stable on some birational model of $X$. They use this result to show that the first dynamical degree of $g$ is either $1$, a Salem number, or a Pisot number. Blanc and Cantat (\cite{BlancCantat2016}) describe the set of Salem and Pisot numbers that arise as dynamical degrees of birational surface transformations. Given the difficulty in computing dynamical degrees, there are several open questions about them. Until recently, it was not known whether every dynamical degree is an algebraic integer: Bell-Diller-Jonsson \cite{BellDillerJonsson2019} have recently found a map with a transcendental dynamical degree. 

It had already been established in \cite{Ramadas2015} that every $R_\phi$ is algebraically stable on the Deligne-Mumford compactifcation, and thus has all dynamical degrees are algebraic integers. Theorem \ref{thm:Main} offers a more refined view, by relating a type of complexity of $R_\phi$ (the length of the periodic cycle of $p_\infty$) to a type of complexity of $R_\phi$ (the degree over $\Q$ of its $k$-th dynamical degree). As a corollary to Theorem \ref{thm:Main}, we obtain:

\begin{cor} 
If $\phi$ is a branched covering with finite post-critical set $\bP$ satisfying criteria (\ref{crit:KC}, \ref{item:KC1}) and (\ref{crit:KC}, \ref{item:KC2}) and such that the fully ramified point $p_\infty$ of (\ref{crit:KC}, \ref{item:KC1}) is in a cycle of period $\ell$, then the $k$th dynamical degree of $R_{\phi}$ is an algebraic integer whose degree over $\Q$ is at most $\rank(H^{k,k}(X_{\ell}))$. 
\end{cor}

The isomorphism class of $X_{\ell}$ depends on $\ell$ and $\abs{\bP}$. For fixed $\abs{\bP}$, fixed $k\in\{1,\ldots,(\dim_{\C}(\M_{0,\bP})-1)\}$ and for $\ell_1,\ell_2\in\{1,\ldots,(\abs{\bP}-2)\}$, if $\ell_1<\ell_2$ then $\rank(H^{k,k}(X_{\ell_1}))<\rank(H^{k,k}(X_{\ell_2}))$. Thus, if we fix $N>0$ the cardinality of post-critical set, and consider
$$\phi\in\{\text{Branched coverings with }\abs{\text{post-critical set}}=N \text{ satisfying (\ref{crit:KC}, \ref{item:KC1}) and (\ref{crit:KC}, \ref{item:KC2})}\},$$ then the shorter the length $\ell$ of the periodic cycle of the fully ramified point of $\phi$, the better an upper bound one can obtain on the degree over $\Q$ of the $k$-th dynamical degree of $R_{\phi}$. More informally:

\medskip
\noindent\textit{If $\phi$ resembles a topological polynomial, then the $k$-th dynamical degree of $R_{\phi}$ is an algebraic integer of low degree over $\Q$.}

\medskip

The sequence of dynamical degrees of any meromorphic map is log-concave \cite{DinhSibony2005}; the sequence of dynamical degrees of a holomorphic map on $\C\P^r$ is log-linear. There is an analysis, in \cite{Ramadas2016}, of how the $k$-th dynamical degree of $R_{\phi}$ depends on $k$. It is shown that the sequence $\{\text{$k$th dynamical degree of $R_{\phi}$}\}_k$ increases strictly with $k$, and that the sequence of logarithms dynamical degrees of $R_{\phi}$ is less concave if $\phi$ resembles a topological polynomial. The precise statement of the result in \cite{Ramadas2016} is very different from the statement of Theorem \ref{thm:Main}, and the proofs are unrelated as well. However, the two statements share the following informal interpretation (generalizing \cite{Koch2013}): 

\medskip

\noindent\textit{If $\phi$ resembles a topological polynomial, then the dynamics of $R_{\phi}$ resemble those of a holomorphic map on $\C\P^{\abs{P}-3}$.}

\medskip

It would be interesting to have a conceptual explanation for the relationship between Theorem \ref{thm:Main} and the results in \cite{Ramadas2016}.

\subsection{Hurwitz correspondences.}
Koch's results in \cite{Koch2013} are more general than described above. Let $\phi:S^2 \to S^2$ be a degree $d$ post-critically finite branched covering with post-critical set $\bP$. If $\phi$ does not satisfy criteria (\ref{crit:KC}, \ref{item:KC1}) and (\ref{crit:KC}, \ref{item:KC2}), then the meromorphic map $R_{\phi}$ need not exist. However, it is still true that the transcendental pullback map $\Th_{\phi}$ induced by $\phi$ on $\mathcal{T}({S^2,\bP})$ descends to an algebraic dynamical system on $\M_{0,\bP}$. However, in general, this is a multivalued map. More precisely, there is an algebraic variety $\H_{\phi}$ admitting a covering map from $\mathcal{T}({S^2,\bP})$ as well as two maps $\pi_1$ and $\pi_2$ to $\M_{0,\bP}$ such that $\pi_1$ is a covering map and the following diagram commutes  \cite{Koch2013}:

\begin{center}
  \begin{tikzpicture}
    \matrix(m)[matrix of math nodes,row sep=1em,column sep=8em,minimum
    width=2em] {
      \mathcal{T}(S^2,\mathbf{P})&&\mathcal{T}(S^2,\mathbf{P})\\
      &\H_{\phi}&\\
      \M_{0,\mathbf{P}}&&\M_{0,\mathbf{P}}\\
    };
    \path[-stealth] (m-1-1) edge node [above] {$\Th(\phi)$} (m-1-3);
    \path[-stealth] (m-1-1) edge (m-2-2);
    \path[-stealth] (m-1-1) edge node [left] {$\substack{\text{universal}\\\text{cover}}$} (m-3-1);
    \path[-stealth] (m-1-3) edge node [left] {$\substack{\text{universal}\\\text{cover}}$} (m-3-3);
    \path[-stealth] (m-2-2) edge node [above left]
    {$\pi_1$} (m-3-1);
    \path[-stealth] (m-2-2) edge node [above right] {$\pi_2$} (m-3-3);
  \end{tikzpicture}
\end{center}

The variety $\H_{\phi}$ is a \emph{Hurwitz space}, a moduli space parametrizing degree $d$ regular maps $f:(\CP^1,\bP)\to(\CP^1,\bP)$, with analogous branching to $\phi$. The Hurwitz space $\H_\phi$ is \textit{non-dynamical} in the sense that it parametrizes maps up to separate changes of coordinates on source and target $\CP^1$; this means that the behavior under iteration of $[f]\in\H_\phi$ is not well-defined. The multivalued map $\pi_2\circ\pi_1^{-1}$ is called a \emph{Hurwitz correspondence}, and considered to be an algebraic ``shadow" of $\Th_{\phi}$. Hurwitz correspondences can be defined purely algebro-geometrically, with no reference to branched coverings of the sphere and to Thurston's pullback map. (Section \ref{sec:HC}; see \cite{RamadasThesis} or \cite{Ramadas2015} for more details). In the case that  $\phi$ satisfies criteria (\ref{crit:KC}, \ref{item:KC1}) and (\ref{crit:KC}, \ref{item:KC2}), Koch showed that $\pi_2$ is generically one-to-one; the meromorphic map $R_{\phi}$ is $\pi_1\circ\pi_2^{-1}$, i.e. a single-valued but meromorphic ``inverse" of the multi-valued but holomorphic Hurwitz correspondence. Thus the graph of $R_{\phi}$ is (up to birational equivalence) the Hurwitz space $\H_{\phi}$.

\subsection{Combinatorial compactifications of moduli spaces and the proof of Theorem \ref{thm:Main}.}\label{sec:CombinatorialCompactifications} The Deligne-Mumford compactification $\Mbar_{0,\bP}$ of $\M_{0,\bP}$ is a moduli space of stable nodal genus zero curves with smooth distinct points marked by elements of $\bP$ (see Section \ref{sec:M0n} for definitions and details). The \emph{boundary} $\Mbar_{0,\bP}\setminus\M_{0,\bP}$ has a combinatorial stratification that is very useful: for example, this stratification is used to give explicit descriptions of the cohomology groups of $\Mbar_{0,\bP}$ \cite{Keel1992}. 

Hassett's \cite{Hassett2003} alternate weighted stable curves compactifications of $\M_{0,\bP}$ parametrize nodal genus zero curves with smooth points --- not necessarily distinct --- marked by elements of $\bP$. Let $\be:\bP\to\Q\cap (0,1]$ be an assignment, to every $p\in\bP$ of some rational `weight', such that the sum of the weights of all $p\in\bP$ is greater than $2$. Then there is a smooth projective compactification $\Mbar_{0,\bP}(\be)$ parametrizing nodal genus zero curves with smooth points marked by $\bP$; the marked points corresponding to a subset of $\bP$ may coincide as long as the sum of the weights of the points in that set is no greater than $1$. The \emph{boundary} $\Mbar_{0,\bP}(\be)\setminus\M_{0,\bP}$ has a combinatorial stratification that is related to the stratification of the boundary of the Deligne-Mumford compactification. Also, $\Mbar_{0,\bP}$ admits a regular birational map $\rho_{\be}$ to $\Mbar_{0,\bP}(\be)$, with fibers that may be positive dimensional over the boundary.

The Hurwitz space $\H_{\phi}$ admits an \emph{admissible covers} compactification $\Hbar_{\phi}$ constructed by Harris and Mumford \cite{HarrisMumford1982}; this compactification extends $\pi_1$ and $\pi_2$ to regular maps to $\Mbar_{0,\bP}$. The map $\pi_1:\Hbar_{\phi}\to\Mbar_{0,\bP}$ is finite and flat; this fact was used in \cite{Ramadas2015} to conclude algebraic stability of all Hurwitz correspondences on  $\Mbar_{0,\bP}$. The boundary of $\Hbar_{\phi}$ has a stratification analogous to and compatible with the stratification of $\Mbar_{0,\bP}$.

Now, suppose $\phi$ satisfies criteria (\ref{crit:KC}, \ref{item:KC1}) and (\ref{crit:KC}, \ref{item:KC2}), so as per the previous section, $\H_{\phi}$ is the graph of the meromorphic map $R_{\phi}$. We set $\bP_\infty\subseteq\bP$ to be the subset of points in the periodic cycle containing the fully ramified point $p_\infty$. We assign weight $\be(p)$ to $p\in\bP$ by the rule  $\be(p)=1$ for $p\in\bP_\infty$ (these are heavy points), and $\be(p)\ll1$ for $p\in(\bP\setminus\bP_\infty)$ (these points are vanishingly light). We formulate and apply a combinatorial analysis of the stratification of $\Hbar_{\phi}$ and of the fibers of $\rho_{\be}$ to show, roughly speaking, that positive-dimensional fibers of $\rho_{\be}\circ\pi_1$ are also positive-dimensional fibers of $\rho_{\be}\circ\pi_2$, and so the meromorphic map $R_{\phi}$ has finite fibers on $\Mbar_{0,\bP}(\be)$. When this analysis is applied to the induced map on cohomology, we obtain the algebraic stability result of Theorem \ref{thm:Main}. 

\subsection{Organization} We begin by introducing $\M_{0,\bP}$ and its compactifications in Section \ref{sec:M0n} and Hurwitz spaces in Section \ref{sec:HC}. In these background sections we frequently refer back to \cite{Ramadas2015}. Section \ref{sec:StaticPoly} contains the key technical contribution of this paper: here, we relate the combinatorics of compactifications of Hurwitz spaces with the combinatorics of certain Hassett spaces. In Section \ref{sec:Proof}, we bring all the ingredients together to prove Theorem \ref{thm:Main}.
\subsection{Conventions}
We work over $\C$.

\subsection*{Acknowledgements}
  I am grateful to my thesis advisors Sarah Koch and David Speyer; this work continues work done during my Ph.D.. I am also grateful to Rob Silversmith for useful conversations, including one that led to a more efficient proof of Lemma \ref{lem:unstabletounstable}, to Melody Chan for useful comments on an earlier draft, and to an anonymous referee for corrections and suggestions that led to significant improvements.

\section{The moduli space $\M_{0,\bP}$ and its compactifications}\label{sec:M0n}
Let $\bP$ be a finite set of cardinality at least $3$. There is a smooth quasiprojective variety $\M_{0,\bP}$ of complex dimension $(\dms)$ that parametrizes injections $\iota:\bP\into\CP^1$ up to the equivalence $\iota\sim\psi\circ\iota$ for any M\"obius transformation $\psi$. The variety $\M_{0,\bP}$ is not compact --- the limit of a one-parameter family of injections $\bP\into\C\P^1$ may irreparably fail to be an injection into $\CP^1$. There are a number of smooth projective compactifications of $\M_{0,\bP}$. The \emph{boundary} of a compactification $X$ is the complement $X\setminus\M_{0,\bP}$. If $X$ is a modular compactification of $\M_{0,\bP}$, i.e. one that extends its interpretation as a moduli space of maps from $\bP$ to an algebraic curve, then points on the boundary of $X$ must correspond to degenerate cases in which either the map is not injective, or the curve has singularities, or both.


\subsection{Stable curves and the Deligne-Mumford compactification $\Mbar_{0,\bP}$.}
The Deligne-Mumford compactification is in some sense the universal and largest modular compactification of $\M_{0,\bP}$: It admits a holomorphic birational map to every other modular compactification (Smyth, \cite{Smyth2009}). 
\begin{Def}
Let $\bP$ be a finite set. A \emph{$\bP$-marked nodal genus zero algebraic curve} is a connected, proper, possibly nodal algebraic curve $C$ of arithmetic genus zero, together with an injection $\iota$ from $\bP$ into the smooth locus of $C$. We say that $(C,\iota)$ is \emph{stable} if $C$ has no nontrivial automorphisms that commute with $\iota$. 
\end{Def}
Concretely, since $C$ has arithmetic genus zero, it is isomorphic to a tree of $\CP^1$s attached at nodes. A \emph{special point} on $C$ is a point of $C$ that is either a node, or in the image of $\iota$. The condition that $(C,\iota)$ have no non-trivial automorphism is equivalent to the condition that every irreducible component of $C$ have at least three special points. For the remainder of this section we suppose that $\bP$ is a finite set of cardinality at least $3$; by works of Deligne, Grothendieck, Knudsen, and Mumford, there is a smooth projective variety $\Mbar_{0,\bP}$ that parametrizes stable $\bP$-marked genus zero algebraic curves, and that contains $\M_{0,\bP}$ as a dense open subset. The boundary $\Mbar_{0,\bP}\setminus\M_{0,\bP}$ has codimension one. Points in the boundary correspond to injective maps from $\bP$ to a nodal algebraic curve; for a general point on the boundary this curve has two irreducible components. The homeomorphism class of a marked nodal curve is encoded combinatorially by a marked tree. For this reason, we introduce below some notation and terminology for describing marked trees and nodal curves. Note that every node on a genus zero curve is disconnecting, in fact, the complement of any node has exactly two connected components.

\begin{Def}\label{def:connect}
Let $(C,\iota)$ be a $\bP$-marked nodal genus zero curve. If $C_{\alpha}$ is an irreducible component of $C$, $x\in C\setminus C_{\alpha}$ and $\eta\in C\setminus\{x\}$ is a node, we say \emph{$\eta$ connects $C_{\alpha}$ to $x$} if $C_{\alpha}\setminus\{\eta\}$ and $x$ are in distinct connected components of $C\setminus\{\eta\}$. If $\eta$ connects $C_{\alpha}$ to $\iota(p)$ for $p\in\bP$, we say that $\eta$ connects $C_{\alpha}$ to $p$. Similarly, if $C_{\alpha}$ and $C_{\beta}$ are two irreducible components of $C$ and $\eta\in C$ is a node, we say \emph{$\eta$ connects $C_{\alpha}$ to $C_{\beta}$} if $C_{\alpha}\setminus\{\eta\}$ and $C_{\beta}\setminus\{\eta\}$ are in distinct connected components of $C\setminus\{\eta\}$. 
\end{Def}

If $x\not\in C_{\alpha}$ for some irreducible component $C_{\alpha}$ of $C$, then there is a unique node $\eta\in C_{\alpha}$ that connects $C_{\alpha}$ to $x$. Similarly, if $C_{\alpha}$ and $C_{\beta}$ are distinct irreducible components, then there is a unique node $\eta$ on $C_{\alpha}$ that connects $C_{\alpha}$ to $C_{\beta}$.

\begin{Def}
  A \emph{$\bP$-marked tree} is a `graph with legs' $\sigma$ defined as follows: $\sigma$ has vertices, edges joining pairs of vertices, and \emph{legs} marked by elements of $\bP$ that are attached to vertices, such that the resulting graph is connected and has no cycles. More formally, $\sigma$ carries the data of: a finite set $\Verts(\sigma)$ of vertices, a finite set $\Edges(\sigma)$ of edges, a map $\Edges(\sigma)\to\Sym^2(\Verts(\sigma))$ encoding the adjacency, a set of legs of $\sigma$ that is canonically identified with $\bP$, and a map $\mrk:\bP\to\Verts(\sigma)$ encoding how the legs are attached. For a vertex $v$ on $\sigma$, set of \emph{flags} on $v$ is defined as: $\Flags_v:=\{\mbox{Legs attached to
  $v$}\}\cup\{\mbox{edges incident to $v$}\}.$
The \emph{valence} of $v$, denoted $\abs{v}$ is defined to be the
cardinality of $\Flags_v.$ We define the \emph{moduli dimension}
$\md(v)$ of $v\in\Verts(\sigma)$ to be $\abs{v}-3$. We say that $\sigma$ is stable if every vertex on $\sigma$ has valence at least $3$, or, equivalently, if very vertex has non-negative moduli dimension. Suppose $\sigma$ is a $\bP$-marked tree, and $v$ is a vertex of $\sigma$. For $p\in\bP$, we define
$\delta(v\to p)$ to be the unique flag in $\Flags_v$ that connects the leg $p$ to $v$, i.e. is part of the unique (non-repeating) path in $\sigma$ from $v$ to $p$. If $\mrk(p)=v$ then $\delta(v\to p)=p;$ otherwise $\delta(v\to p)$ is an edge. Similarly, for $v_1$ and $v_2$ two distinct vertices of $\sigma$,  we define
$\delta(v_1\to v_2)$ to be the unique flag in $\Flags_{v_1}$ that is part of the path in
$\sigma$ from $v_1$ to $v_2$.  
 \end{Def}

\begin{Def}\label{Def:DualTree}
  Let $(C,\iota)$ be a $\bP$-marked nodal genus zero curve. Its
  \emph{dual tree} is the $\bP$-marked tree $\sigma$ defined as follows. The
  vertices $v$ of $\sigma$ correspond to irreducible components
  $C_v$ of $C$. Two vertices $v_1$ and $v_2$ are joined by an edge
  if and only if the components $C_{v_1}$ and $C_{v_2}$ intersect at a node. Thus nodes of $C$ correspond to edges of $\sigma$. For each marked point $\iota(p)$ on $C_v,$ we attach a
  leg marked by $p$ to the vertex $v$, i.e. $\mrk(p)=v$. The graph $\sigma$ has no loops because
  $C$ has arithmetic genus zero. Note that $\sigma$ is stable if and only is $(C,\iota)$ is.  
\end{Def}

For fixed $\bP$, there are finitely many isomorphism classes of stable
$\bP$-marked trees, and each of these arises as the dual tree of some $\bP$-marked stable genus zero curve. The classification of stable curves by topological type gives a
stratification of $\Mbar_{0,\bP}$. 

%
\begin{Def}\label{Def:BoundaryStratum}
  Given $\sigma$ a stable $\bP$-marked tree, the closure $S_\sigma$ of
  the locus $S^{\circ}_{\sigma}$ of curves with dual graph $\sigma$ is an irreducible subvariety of $\Mbar_{0,\bP}$ isomorphic to
 \begin{align}
\prod_{v\in\Verts(\sigma)}\Mbar_{0,\Flags_v} \label{eq:product2}
\end{align}
 We refer to $S_{\sigma}$ as a \emph{boundary stratum} of $\Mbar_{0,\bP}$; Boundary strata on $\Mbar_{0,\bP}$ are in bijection with isomorphism classes of stable $\bP$-marked trees.  
\end{Def}

From the above decomposition \ref{eq:product2} of $S_{\sigma}$ into a product we obtain that the dimension of a boundary stratum $S_{\sigma}$ is 
 \begin{align}
 \sum_{v\in\Verts(\sigma)}\dim_{\C}(\Mbar_{0,\Flags_v}) = \sum_{v\in\Verts(\sigma)}\md(v) \label{eq:dimensionofstratum}
\end{align}

\subsection{Stabilization and forgetful maps.} \label{sec:forget}
Suppose $\abs{\bP}\ge3$ and $(C, \iota)$ is a $\bP$-marked nodal genus zero curve. Then there is a unique curve $C'$, together with a surjective map $\st:C\to C'$, such that $(C', \st\circ\iota)$ is stable. The curve $C'$ is called the \emph{stabilization} of $C$, and is obtained from $C$ as follows. Let $\sigma$ be the dual tree of $C$. Given an irreducible component $C_v$ of $C$ corresponding to vertex $v$ of $\sigma$, we say that $C_v$ (resp. $v$) is \emph{$\bP$-stable} if there are at least three special points on $C_v$ of the form either a marked point $\iota(p)$ or  a node $\eta$ that connects $C_v$ to some marked point $p$. This is equivalent to the condition that there are at least three flags on $v$ of the form $\delta(v\to p)$ for some $p\in\bP$. We obtain $C'$ from $C$ by contracting to a point each connected component of the closure of $C\setminus \bigcup_{v \text{ $\bP$-stable}} C_v.$ The map $\st:C\to C'$ is the resulting map: a component $C_v$ maps isomorphically onto its image in $C'$ if and only if it is stable; otherwise $\st(C_v)$ is a point. 



Now, let $j:\bP'\into\bP$ be an injection of finite sets, where $\abs{\bP},\abs{\bP'}\ge 3$. There is a \emph{forgetful map} $\mu:\M_{0,\bP}\to\M_{0,\bP'}$ sending $[(\CP^1,\iota)]$ to $[(\CP^1,\iota\circ j)]$. If $(C, \iota)$ is a $\bP$-marked stable curve, then $(C, \iota\circ j)$ is not necessarily stable. However, we can obtain from $(C, \iota\circ j)$ a stable curve by stabilizing as described above. In this way, $\mu$ extends to a regular map from $\Mbar_{0,\bP}$ to $\Mbar_{0,\bP'}$. If $\sigma$ is a $\bP$-marked stable tree, forgetting the points in $\bP\setminus\bP'$ yields a $\bP'$-marked tree, in general not stable. We have:

\begin{lem}\label{lem:forget}
If $S_\sigma$ is a boundary stratum of $\Mbar_{0,\bP}$, then $\mu(S_\sigma)$ is a boundary stratum of $\Mbar_{0,\bP'}$, and the restriction of $\mu$ to $S_\sigma$ factors through the projection:
\begin{align}
S_{\sigma}=\prod_{v\in\Verts(\sigma)}\Mbar_{0,\Flags_v} \to \prod_{\substack{v\in\Verts(\sigma)\\\text{$\bP'$-stable}}}\Mbar_{0,\Flags_v}, \label{eq:forgetofstratum}
\end{align}
\end{lem}

\subsection{Hassett spaces/Moduli spaces of weighted stable curves.}\label{sec:WeightedStableCurves}

These are alternate compactifications of $\M_{0,\bP}$ constructed by Hassett in \cite{Hassett2003}. Points in the boundary of these compactifications parametrize possibly nodal curves $C$ that are marked by elements of $\bP$; but here the marked points are assigned rational weights that prescribe the extent to which they are allowed to coincide. 

\begin{Def}
  A \emph{weight datum} on $\M_{0,\bP}$ is a map  $\be:\bP\to\Q\cap(0,1]$
  such that $\sum_{p\in\bP} \be(p)>2.$
\end{Def}

\begin{Def}  Let $\be$ be a weight datum on $\M_{0,\bP}$. A \emph{$\bP$-marked $\be$-stable genus zero curve} is a possibly nodal curve $C$ of arithmetic genus zero, together with a (not necessarily injective) map $\mpt:\bP\to (\text{smooth locus of }C)$, such that  
\begin{enumerate}
  \item If $\mpt(p_{1})=\cdots=\mpt(p_{s}$) then
    $\be(p_1)+\cdots+\be(p_s)\le1$, and \label{item:wtle1}
  \item For every irreducible component $C_v$,
    \begin{align*}
      \#\{\mbox{nodes on $C_v$}\}+\sum_{\{p|\mpt(p)\in C_v\}}\be(p)>2. 
    \end{align*} \label{item:wtge1}
  \end{enumerate}
\end{Def}
Like a stable curve, $C$ is isomorphic to a tree of $\CP^1$s attached at nodes, and is marked by elements of $\bP$. Condition (\ref{item:wtle1}) specifies that marked points may coincide if their combined weights don't exceed one. Condition (\ref{item:wtge1}) ensures that any node on $C$ partitions the set $\bP$ into two sets, both of which have total weight greater than one. 

\begin{Defthm}[Hassett, \cite{Hassett2003}]\label{thm:Hassett}
  \begin{enumerate}
  \item Given a weight datum $\be$ on $\M_{0,\bP}$, there is a smooth projective
  variety $\Mbar_{0,\bP}(\be)$ that parametrizes $\bP$-marked $\be$-stable genus zero curves and
  contains $\M_{0,\bP}$ as a dense open set. 
  \item There is a
  \emph{reduction map}
  $\rho_{\be}:\Mbar_{0,\bP}\to\Mbar_{0,\bP}(\be)$ that
  respects the open inclusion of $\M_{0,\bP}$ into both spaces.
  \item If $\be_1$ and $\be_2$ are two weight data on $\M_{0,\bP}$ such that for every $p\in\bP$, $\be_1(p)\ge\be_2(p)$, then there is a \emph{generalized reduction map} $\rho_{\be_1, \be_2}:\Mbar_{0,\bP}(\be_1)\to\Mbar_{0,\bP}(\be_2)$ such that $\rho_{\be_2}=\rho_{\be_1, \be_2}\circ\rho_{\be_1}$.
    \end{enumerate}
\end{Defthm}

\begin{ex} \label{ex:HassettSpace}
\begin{enumerate}
\item Set $\be(p)=1$ for all $p\in\bP$; this is a weight datum as long as $\abs{\bP}\ge 3$. Then the notions of stability and $\be$-stability coincide, so $\Mbar_{0,\bP}\isom\Mbar_{0,\bP}(\be)$. Thus the Deligne-Mumford compactification is a special case of a Hassett space.
\item Fix $p_{\infty}\in\bP$, and $\epsilon\in\Q$ such that $(\abs{\bP}-2)<(1/\epsilon)$ but $(\abs{\bP}-1)>(1/\epsilon)$. Set $\be(p_{\infty})=1$, and $\be(p)=\epsilon$ for all $p\ne p_{\infty}$. Then $\Mbar_{0,\bP}(\be)\isom\CP^{\abs{\bP}-3}$. \label{it:ProjAsHassett}
\end{enumerate}
 \end{ex}

The reduction map $\rho_{\be}$ can be described explicitly: Suppose $(C, \iota)$ is a $\bP$-marked nodal genus zero curve. Then there is a unique curve $C'$, together with a surjective map $\st_{\be}:C\to C'$, such that $(C', \st_{\be}\circ\iota)$ is $\be$-stable. The curve $C'$ is called the \emph{$\be$-stabilization} of $C$, and is obtained from $C$ as follows. Let $\sigma$ be the dual tree of $C$. 

\begin{Def}\label{def:epsilonverystable}
Given an irreducible component $C_v$ of $C$ corresponding to vertex $v$ of $\sigma$, we say that $v$ (resp. $C_v$) is \emph{$\be$-stable} if:
 \begin{align}
    \sum_{\fl\in\Flags_v}\min\left\{1,\sum_{\{p|\fl=\delta(v\to
        p)\}}\be(p)\right\}>2.
  \end{align}
  \end{Def}
  
We obtain $C'$ by contracting to a point each connected component of the closure of
$(C\setminus \bigcup_{(v \text{ $\be$-stable)}} C_v)$. The induced stabilization map $\st_{\be}:C\to C'$ maps a component $C_v$ isomorphically onto its image in $C'$ if and only if $C_v$ is $\be$-stable; otherwise $\st_{\be}(C_v)$ is a point. The reduction map $\rho_{\be}:\Mbar_{0,\bP}\to\Mbar_{0,\bP}(\be)$ sends $[(C,\iota)]\in\Mbar_{0,\bP}$ to its $\be$-stabilization. Given a boundary stratum $S_{\sigma}$ of $\Mbar_{0,\bP}$, we consider the natural projection map
\begin{align}
S_{\sigma}=\prod_{v\in\Verts(\sigma)}\Mbar_{0,\Flags_v} \to \prod_{\substack{v\in\Verts(\sigma)\\\text{$\be$-stable}}}\Mbar_{0,\Flags_v}=:\PS_{\sigma}, \label{eq:projofstratum}
\end{align}
where $\PS_{\sigma}$ is defined to be the product on the right. We obtain from the above description the following lemma.

\begin{lem}\label{lem:discardunstable}
Let $S_{\sigma}\subset\Mbar_{0,\bP}$ be a boundary stratum. Then: 
\begin{enumerate}
\item the restriction $\rho_{\be}|_{S_{\sigma}}$ factors through the projection in equation (\ref{eq:projofstratum}),
\item the induced map from $\PS_{\sigma}$ to $\Mbar_{0,\bP}(\be)$ is birational onto its image, and
\item 
$\dim_{\C}(\rho_{\be}(S_{\sigma}))=\sum_{\substack{v\in\Verts(\sigma)\\\text{$\be$-stable}}}\md(v)$
\end{enumerate}
\end{lem}

\subsection{Heavy/light Hassett spaces.} In this paper we will be primarily concerned with a certain subclass of spaces of weighted stable curves. These spaces are called \emph{heavy/light} Hassett spaces and have appeared in studies of tropical moduli spaces of curves \cite{CHMRHeavyLight, KannanKarpLi2019}.
\begin{Def}
Suppose $\abs{\bP}\ge3$, and there is a decomposition $\bP=\bP_{\hvy}\sqcup\bP_{\lt}$ with $\abs{\bP_{\hvy}}\ge2$. Let $\epsilon>0$ be any rational number such that $\abs{\bP_{\lt}}<(1/\epsilon)$. Then the weight datum $\be$ sending $p\in\bP_{\hvy}$ to $1$ (these are the heavy points) and $p\in\bP_{\lt}$ to $\epsilon$ (these are the light points) is called a heavy/light weight datum, and the resulting moduli space $\Mbar_{0,\bP}(\be)$ is called a heavy/light Hassett space. 
\end{Def}

Heavy/light weight data $\be$ can be characterized in the following manner: on a $\be$-stable curve, heavy marked points may not coincide with each other or with light marked points, but light marked points may coincide with each other to an arbitrary extent. Thus the isomorphism class of the moduli space $\Mbar_{0,\bP}(\be)$ does not depend on the value of the rational number $\epsilon$; it depends only on the numbers of heavy and light points. If the number of light points is one or two, then the resulting heavy/light space is isomorphic to the Deligne-Mumford compactification. The number of heavy points must be at least $2$; if that number is exactly $2$, the resulting heavy/light space is a toric variety called a \emph{Losev-Manin} space and has been studied independently.

Since we will be interested in understanding the reduction maps from $\Mbar_{0,\bP}$ to various heavy/light Hassett spaces, the following characterization of $\be$-stableness for heavy/light data will be useful.

%


\begin{lem}\label{lem:HvyLtStable}
Suppose $\be$ is a heavy/light weight datum, with $\bP=\bP_{\hvy}\sqcup\bP_{\lt}$. Then
\begin{enumerate} 
\item (Statement about curves.) \label{it:HvyLtStableCurves}
\begin{enumerate}
\item An irreducible component $C_1$ of a nodal $\bP$-marked curve $(C, \iota)$ is not $\be$-stable if 
\begin{align}
\{p\in \bP_{\hvy} \text{ s.t. } \iota(p)\in C_1\}\cup\{\eta\in C_1 \text{ node connecting $C_1$ to some $p\in \bP_{\hvy}$} \}\label{eq:goodset1}
\end{align}
has cardinality one or less.\label{it:HvyLtStableCurves1}
\item An irreducible component $C_1$ of a stable $\bP$-marked curve $(C, \iota)$ is not $\be$-stable if and only if the set in (\ref{eq:goodset1}) has cardinality one.\label{it:HvyLtStableCurves2}
\end{enumerate}
\item (Equivalent statement about trees.) \label{it:HvyLtStableTrees}
\begin{enumerate}
\item A vertex $v$ of a $\bP$-marked tree $\sigma$ is not $\be$-stable if 
\begin{align}
\{\fl \in\Flags_v|\fl=\delta(v\to p) \text{ for some }p\in\bP_{\hvy}\}\label{eq:goodset2}
\end{align}
has cardinality one or less. \label{it:HvyLtStableTrees1}
\item A vertex $v$ of a stable $\bP$-marked tree $\sigma$ is not $\be$-stable if and only if the set in (\ref{eq:goodset2}) has cardinality one. \label{it:HvyLtStableTrees2}
\end{enumerate}
\end{enumerate}
\end{lem}
\begin{proof}
Since the equivalence of items (\ref{it:HvyLtStableCurves}) and (\ref{it:HvyLtStableTrees}) is clear, we prove only (\ref{it:HvyLtStableTrees}). First, we claim that for any vertex $v$ on a $\bP$-marked tree $\sigma$, the cardinality of set in (\ref{eq:goodset2}) must be at least one: Since there exists some $p_0\in \bP_{\hvy}$, and $\sigma$ is connected, $\exists \fl_0\in\Flags_v$ connecting $v$ to $p_0$. 

Now suppose $v$ is a vertex of $\sigma$ such that the set in (\ref{eq:goodset2}) has cardinality one. Then  
\begin{align}
    \sum_{\fl\in\Flags_v}\min\left\{1,\sum_{\{p|\fl=\delta(v\to
        p)\}}\be(p)\right\}=\sum_{\substack{\fl\in\Flags_v \\ \fl=\delta(v\to p) \\\text{ for some }\\ p\in\bP_{\hvy}}} 1+ \sum_{\substack{\fl\in\Flags_v \\ \fl\ne\delta(v\to p) \\\text{ for any }\\p\in\bP_{\hvy}}} \sum_{\substack{p \text{ s.t.} \\\fl=\delta(v\to p)}} \epsilon \le 1+ \abs{\bP_{\lt}}\epsilon <2.
  \end{align}
  So, according to Definition \ref{def:epsilonverystable}, $v$ is not $\be$-stable, proving part (\ref{it:HvyLtStableTrees1}). 
  
Finally, we suppose that $\sigma$ is a stable $\bP$-marked tree, and $v$ on $\sigma$. Since $\sigma$ is stable, $v$ is $\bP$-stable, so there are at least three flags on $v$ of the form $\delta(v\to p)$ for some $p\in\bP$. If there are two or more flags on $v$ for the form $\delta(v\to p)$ for some $p\in\bP_{\hvy}$, then 
\begin{align}
    \sum_{\fl\in\Flags_v}\min\left\{1,\sum_{\{p|\fl=\delta(v\to
        p)\}}\be(p)\right\}\ge 2+\epsilon>2,
\end{align}
so $v$ is $\be$-stable. So if $v$ is not $\be$-stable, then the set (\ref{eq:goodset2}) has cardinality exactly one, proving the lemma. 
\end{proof}

\subsection{A tower of compactifications}\label{sec:tower}

Let $\abs{\bP}\ge3$, and suppose $\epsilon>0$ is such that $(\abs{\bP}-2)<(1/\epsilon)$ and $(\abs{\bP}-1)>1/\epsilon$. Fix $p_{\infty}\in \bP$ and subsets of $\bP$

$$\{p_{\infty}\}=\bP_1\subset\bP_2\subset\bP_3,\cdots,\subset\bP_{\abs{\bP}}=\bP$$

such that $\abs{\bP_\ell}=\ell$. For $\ell=1,\ldots,\abs{\bP}$, let $X_\ell$ be the Hassett space corresponding to the weight datum assigning the points in $\bP_\ell$ weight $1$ and all other points weight $\epsilon$. For $\ell=2,\ldots,\abs{\bP}$, $X_\ell$ is a heavy/light Hassett space with $\ell$ heavy points and $(\abs{\bP}-\ell)$ light points. As stated in the previous section, $\Mbar_{0,\bP}\isom X_{\abs{\bP}}\isom X_{\abs{\bP}-1}\isom X_{\abs{\bP}-2}$, and $X_2$ is a Losev-Manin space. On the other hand, $X_1$ is not a heavy/light space: $X_1\isom\CP^{\abs{\bP}-3}$, as described in Example \ref{ex:HassettSpace}, (\ref{it:ProjAsHassett}). By Theorem \ref{thm:Hassett}, we have reduction maps $\rho_{\ell+1,\ell}:X_{\ell+1}\to X_{\ell}$ for $\ell=1,\ldots,\abs{\bP}-1$. These are the spaces and maps referred to in the statement of Theorem \ref{thm:Main}.

\subsection{(Co)homology groups of compactifications of $\M_{0,\bP}$.}

In this work, we only consider the Deligne-Mumford compactifications and the Hassett weighted stable curves compactification of $\M_{0,\bP}$. For any such compactification $X_{\bP}$ we have \cite{Keel1992, Ceyhan2009} that  $H_{2k}(X_{\bP},\Z)$ is a finitely generated free abelian group generated by fundamental classes of boundary strata. We also have identifications $H^{2k}(X_\bP, \Z)=H_{2(\dim(X_{\bP})-k)}(X_\bP, \Z)$ and $H^{2k}(X_\bP, \R)=H^{k,k}(X_{\bP})$. A boundary stratum in $\Mbar_{0,\bP}(\be)$ is the image, under $\rho_{\be}$, of a boundary stratum in $\Mbar_{0,\bP}$. This tells us that $H_{2k}(\Mbar_{0,\bP}(\be),\Z)$ is the quotient of $H_{2k}(\Mbar_{0,\bP},\Z)$ by the kernel of the pushforward $(\rho_{\be})_*$. By \cite{Ramadas2015} (Lemma 10.9), the kernel of $(\rho_{\be})_*:H_{2k}(\Mbar_{0,\bP}(\be),\Q)\to H_{2k}(\Mbar_{0,\bP},\Q)$ is generated by fundamental classes of boundary strata. Lemma \ref{lem:discardunstable} allows us to describe $\ker((\rho_{\be})_*)$ as follows:
\begin{lem}\label{lem:CriterionForKernel}
\begin{enumerate}
\item The pushforward $(\rho_{\be})_*([S_\sigma])$ is nonzero if and only if
  every vertex $v\in\Verts(\sigma)$ with positive moduli dimension is
  $\be$-stable.
\item $\ker((\rho_{\be})_*)=\Span\{[S_\sigma] \text{ $k$-dim }| \quad\exists v\in\Verts(\sigma) \text{ not $\be$-stable with }\md(v)>0\}.$
\end{enumerate}
\end{lem}

\subsection*{Change of notation} In the subsequent sections, for a $\bP$-marked curve $(C, \iota)$ or $(C, \mpt)$, we will suppress the notation  $\iota$/$\mpt$ for the marking map, and just write $(C,\bP)$. 

\section{Hurwitz correspondences}\label{sec:HC}

Hurwitz spaces are moduli spaces parametrizing finite maps with
prescribed ramification between smooth curves. We refer the reader to 
\cite{RomagnyWewers2006} for a general summary and to \cite{Ramadas2015} for the definitions as used in this paper. In particular, we use Definition 5.4 of \cite{Ramadas2015} for our definition of the \emph{Hurwitz space}: Fix discrete data: $\bA$ and $\bB$ finite sets, $d\in\Z^{>0}$, $F:\bA\to \bB$, $\br:\bB\to\{\mbox{partitions of $d$}\}$, and $\rm:\bA\to\Z^{>0}$. Then $\cH=\cH(\bA,\bB,d,F,\br,\rm)$ is a smooth quasiprojective variety parametrizing morphisms $f:(C,\bA)\to (D,\bB)$, where $C$ and $D$ are, respectively, $\bA$-marked and $\bB$-marked smooth connected genus zero curves, $f$ is degree $d$, and maps the points in $\bA$ to those in $\bB$ as specified by $F$, with ramification at points in $\bA$ and branching over points in $\bB$ as specified by $\rm$ and $\br$ respectively. The Hurwitz space $\cH$ has a ``source curve'' map $\pi_{\bA}$ to $\M_{0,\bA}$ sending $[f:(C,\bA)\to (D,\bB)]$ to the marked curve $[(C,\bA)]$. There is
similarly a ``target curve'' map $\pi_{\bB}$ from $\H$ to
$\M_{0,\bB}$. Unless $\cH$ is empty, $\pi_{\bB}$ is a finite covering
map. Thus the triple $(\cH,\pi_{\bB},\pi_{\bA}):\M_{\bB}\multi
\M_{\bA}$ is a multi-valued map. We generalize this notion.
\begin{Def}[\emph{Hurwitz correspondence}, \cite{Ramadas2015},
  Definition 5.5]
  With notation as above, let $\bA'$ be any subset of $\bA$ with cardinality at least 3. There
  is a forgetful map $\mu:\M_{0,\bA}\to\M_{0,\bA'}$. Let $\Gamma$
  be a union of connected components of $\cH$. We call the triple $(\Gamma, \pi_{\bB},\mu\circ\pi_{\bA}):\M_{0,\bB}\multi \M_{0,\bA'}$ a \emph{Hurwitz correspondence}.
\end{Def}

\subsection{Hurwitz correspondences and meromorphic maps from PCF maps}\label{sec:HurwitzfromPCF}
Suppose $\phi:S^2\to S^2$ is a degree $d$ orientation-preserving branched covering with finite post-critical set $\bP$. Define $\rm:\bP\to\Z^{>0}$ sending $p\in\bP$ to the local degree of $\phi$ at $p$. Define $\br:\bP\to\{\text{partitions of $d$}\}$ sending $p\in\bP$ to the branching profile of $\phi$ over $p$.  Then $\H=\H(\bP,\bP,d,\phi|_{\bP},\br,\rm)$ parametrizes regular maps $(\C\P^1,\bP)\to(\C\P^1,\bP)$ with the same branching as $\phi$. Let $\pi_1$ and $\pi_2$ be the ``target" and ``source" maps from $\H$ to $\M_{0,\bP}$. There is a unique connected component $\H_{\phi}$ of $\H$ parametrizing maps that are topologically isomorphic to $\phi$, i.e. maps $f:(\C\P^1,\bP)\to(\C\P^1,\bP)$ such that there exist marked-point-preserving homeomorphisms $\chi_1$ and $\chi_2$ from $(\C\P^1,\bP)$ to $(S^2, \bP)$ with $\chi_2\circ f= \phi\circ\chi_1$. By \cite{Koch2013}, the Hurwitz correspondence $(\H_{\phi},\pi_1,\pi_2)$ on $\M_{0,\bP}$ is descended from the Thurston pullback map $\Th_{\phi}$. When, in addition, $\phi$ satisfies criteria \ref{crit:KC}, \ref{item:KC1} and \ref{crit:KC}, \ref{item:KC2}, Koch showed that $\pi_2:\H_{\phi}\to\M_{0,\bP}$ is generically one-to-one; the meromorphic map $R_{\phi}$ is $\pi_2\circ\pi_1^{-1}$. Thus the graph of $R_{\phi}$ is (up to birational equivalence) the Hurwitz space $\H_{\phi}$, i.e. the following diagram commutes:

\begin{center}
\begin{tikzcd}[row sep=20pt, column sep=20pt]
  &\H_\phi\arrow[dl,"\pi_1",swap]\arrow[dr,shift left=3pt,"\pi_2"]&\\
  \M_{0,\mathbf{P}}&&\M_{0,\mathbf{P}}\arrow[ll,dashed,"R_\phi"]\arrow[ul,dashed,shift
  left=3pt,"\pi_2^{-1}"]
\end{tikzcd}
\end{center}

As described in \cite{DinhSibony2008, Truong2016, Ramadas2015}, correspondences can be composed and dynamical correspondences such as $(\H_{\phi},\pi_1,\pi_2)$ can be iterated. When the meromorphic map $R_{\phi}$ exists, then the multivalued map $(\H_{\phi}^n,\pi_{1,n},\pi_{2,n})$ given by the $n$-th iterate of $(\H_{\phi},\pi_1,\pi_2)$ is the inverse of $R_{\phi}^n$.

\subsection{The maps on (co)homology induced by Hurwitz correspondences}\label{sec:Hurwitzcohomology}

Suppose $(\Gamma, \pi_{\bB},\mu\circ\pi_{\bA}):\M_{0,\bB}\multi \M_{0,\bA'}$ is a Hurwitz correspondence from $\M_{0,\bB}$ to $\M_{0,\bA'}$, $X_{\bB}$ is some smooth projective compactification of $\M_{0,\bB}$ and $X_{\bA'}$ is some smooth projective compactification of $\M_{0,\bA'}$. Then there are induced pushforward maps on homology groups, (and, dually, pullback maps on cohomology groups) as follows. Let $\bar{\Gamma}$ be any smooth projective compactification of $\Gamma$ such that the maps $\pi_{\bB}:\bar{\Gamma}\to X_{\bB}$ and $(\mu\circ\pi_{\bA}):\bar{\Gamma}\to X_{\bA'}$ are both regular. Then $[\Gamma]_*:=(\mu\circ\pi_{\bA})_*\circ\pi_{\bB}^*:H_{2k}(X_{\bB}, \Z)\to H_{2k}(X_{\bA'}, \Z)$. The pushforward and pullback maps are well-defined, i.e. they do not depend on the choice of compactification $\bar{\Gamma}$, but they are not in general functorial with respect to composition of correspondences. (See \cite{DinhSibony2008, Ramadas2015} for details). However, the maps induced by Hurwitz correspondences on the (co)homology groups of the Deligne-Mumford compactifications in particular are functorial with respect to composition \cite{Ramadas2015}. Now suppose $\phi:S^2\to S^2$ is PCF and satisfies criteria (\ref{crit:KC}, \ref{item:KC1}) and  (\ref{crit:KC}, \ref{item:KC2}) so $R_{\phi}$ exists. By the definition of pullback by a meromorphic map given in \cite{Roeder2013}, and the fact that $\H_{\phi}$ is the graph of $R_{\phi}$, we have, for any compactification $X_{\bP}$ of $\M_{0,\bP}$, and $\forall n>0$, that $[\H_{\phi}^n]_*=(R_{\phi}^n)^*$ and $[\H_{\phi}^n]^*=(R_{\phi}^n)$ on $H_{2k}(X_{\bP},\Z)$. This implies that by \cite{Ramadas2015}, $R_{\phi}$ is algebraically stable on $\Mbar_{0,\bP}$.

\subsection{Compactifications of Hurwitz spaces}\label{sec:adm}

An \emph{admissible cover} is a ramified map between nodal curves that satisfies a certain balancing condition at nodes. Harris and Mumford \cite{HarrisMumford1982} defined admissible covers and constructed their moduli spaces, which compactify Hurwitz spaces. We refer the reader to Definition 7.3 of \cite{Ramadas2015} for the definition of a \emph{$(\bA,\bB,d,F,\br,\rm)$-admissible cover}, as it is used here. In general, the admissible covers compactifications are only coarse moduli spaces with orbifold singularities. For technical ease, \cite{Ramadas2015}, Definition 7.1 introduces \emph{fully marked Hurwitz spaces}, a class of Hurwitz spaces parametrizing maps of curves $f:(C,\bA)\to (D,\bB)$ with the property that $\forall a\in C$ with $f(a)\in\bB$, we must have $a\in \bA$. In other words, a point on the source curve is marked if (and only if) its image on the target curve is marked. The admissible covers compactifications of fully marked Hurwitz spaces are fine moduli spaces. 

\begin{thm}[Harris and Mumford, \cite{HarrisMumford1982}]
  Given $\H=(\bA,\bB,d,F,\br,\rm)$ a fully marked Hurwitz space as in
  \cite{Ramadas2015}, Definition 7.1, there is a projective variety
  $\Hbar=\Hbar(\bA,\bB,d,F,\br,\rm)$ parametrizing admissible covers, and containing
  $\H=\H(\bA,\bB,d,F,\br,\rm)$ as a dense open subset. The
  compactification $\Hbar$ extends the maps
  $\pi_\bB$ and $\pi_\bA$ to maps $\bar{\pi_{\bB}}$ and
  $\bar{\pi_{\bA}}$ to $\Mbar_{0,\bB}$ and $\Mbar_{0,\bA}$,
  respectively, with $\bar{\pi_\bB}:\Hbar\to\Mbar_{0,\bB}$ a finite
  flat map. $\Hbar$ may not be normal, but its normalization is
  smooth.
\end{thm}
\begin{rem}
  The irreducible components of $\Hbar$ are the Zariski closures of the connected
  components of $\H.$
\end{rem}

\subsection{Boundary strata in $\Hbar$}\label{sec:StratificationOfAdmissibleCovers}
Moduli spaces of admissible covers have a stratification analogous to
and compatible with that of $\Mbar_{0,n}$. This stratification has been studied in detail in \cite{CavalieriMarkwigRanganathan2016}. In this section we fix $\H=\H(\bA,\bB,d,F,\br,\rm)$ a fully marked Hurwitz space, and let
$\Hbar=\Hbar(\bA,\bB,d,F,\br,\rm)$. If $[f:(C,\bA)\to (D,\bB)]\in\Hbar$ is an admissible cover, then there is an induced map of graphs from the dual tree of $C$ to that of $D$, as described in \cite{Caporaso2014}. The \emph{combinatorial type} of an admissible cover records this map of graphs together with other discrete data that describe how the irreducible components of $C$ map to those of $D$. We refer the reader to \cite{CavalieriMarkwigRanganathan2016} for the general definition of combinatorial type of admissible cover, and to Definition 7.8 of \cite{Ramadas2015} for the specific definition, used here, of the combinatorial type $\gamma$ of $[f:(C,\bA)\to (D,\bB)]\in\Hbar$, where:

\begin{align} \label{eq:combinatorialtype}
  \gamma=(\sigma,\tau,d_{\Verts},f_{\Verts},F_{\Edges},(\br_v)_{v\in\Verts(\sigma)},\rm_{\Edges}).
\end{align}

Here, $\sigma$ and $\tau$ denote the dual trees of $C$ and $D$ respectively, $f_{\Verts}$ and $F_{\Edges}$ record, respectively, how the irreducible components and the nodes of $C$ map to those of $D$, $d_{\Verts}$ records the degrees of the restrictions of $f$ to the irreducible components of $C$, and $(\br_v)_{v\in\Verts(\sigma)}$ and $\rm_{\Edges}$ record, respectively, the branching of $f$ over nodes of $D$ and at nodes of $C$.

\begin{Def}
  The closure $G_\gamma$ of $\{f':C'\to D'|\mbox{$f'$ has
    combinatorial type $\gamma$}\}$ is a subvariety of $\Hbar$. We
  call such a subvariety a \emph{boundary stratum} of $\Hbar$.
\end{Def}

The boundary stratum $G_\gamma$ in $\Hbar$ can be decomposed into a
product of lower-dimensional spaces of admissible covers. For
$v\in\Verts(\sigma),$ the Hurwitz space $\H(\Flags_v,\Flags_{f_{\Verts}(v)},d^\vert(v),F_{\Edges},\br_v,\rm_{\Edges})$ is fully marked. Denote by $\Hbar_v$ the corresponding space of admissible covers; this is reducible in general. The space
$\Hbar_v$ admits maps to the moduli space $\Mbar_{0,\Flags_v}$ of
source curves and the moduli space $\Mbar_{0,\Flags_{f_{\Verts}(v)}}$
of target curves. For $w\in\Verts(\tau),$ set $\Hbar_w:=\prod_{v\in(f_{\Verts})^{-1}(w)}\Hbar_v$, where the product is fibered over the common moduli space $\Mbar_{0,\Flags_w}$ of target curves. The fibered product $\Hbar_w$ is itself a moduli space of possibly
disconnected admissible covers, admitting a map
$\bar\pi_w^{\mathrm{source}}$ to the moduli space
$\prod_{v\in(f_{\Verts})^{-1}(w)}\Mbar_{0,\Flags_v}$ of source curves
and a finite flat map $\bar\pi_w^{\mathrm{target}}$ to the moduli
space $\Mbar_{0,\Flags_w}$ of target curves. The stratum $G_\gamma$ is
isomorphic to $\prod_{w\in\Verts(\tau)}\Hbar_w.$ The boundary stratum $G_\gamma$ is not necessarily irreducible. Its irreducible components are of the form 
\begin{align}\label{eq:BdryCompAsProd}
J=\prod_{w\in\Verts(\tau)}\bar{\cJ_w} 
\end{align}
 where
$\bar{\cJ_w}$ is an irreducible component of $\Hbar_w.$

\section{Moduli spaces of static polynomials}\label{sec:StaticPoly}

A polynomial $f[z]$ in one variable defines a regular map $f:\CP^1\to\CP^1$ for which $\infty\in \C\P^1$ is a fully ramified fixed point. More generally, a regular map $f':\CP^1\to\CP^1$ is called a \emph{polynomial} if there is a fully ramified fixed point $a_\infty\in \C\P^1$; such a map $f'$ is conjugate to a regular map defined by a polynomial in one variable. We recall from Section \ref{sec:Intro} that a topological polynomial is a branched covering $\phi:S^2 \to S^2$ that has a fully ramified fixed point. The condition of having a fully ramified point is invariant under separate changes of coordinates on source and target, i.e. it is a \emph{non-dynamical/static} feature. On the other hand, the condition of being a fixed point is invariant under the same change of coordinates on source and target, but not invariant under separate changes of coordinates on source and target. In other words, the condition of being a fixed point is a \emph{dynamical} feature. Now, suppose $\phi$ satisfies criteria (\ref{crit:KC}, \ref{item:KC1}) and (\ref{crit:KC}, \ref{item:KC2}). Then, although $\phi$ may not be a topological polynomial, it shares a non-dynamical feature with topological polynomials, i.e. there is a point $p_\infty$ that is a fully ramified point of $\phi$, although it may not be fixed. This motivates the following definition:

\begin{Def}
We say that a regular map $f:\CP^1\to\CP^1$ of degree $d$ is a \emph{static polynomial} if it has a fully ramified point, i.e. if exists $a_{\infty}\in\CP^1$ such that the local degree of $f$ at $a_\infty$ is $d$. Similarly, we say that a degree $d$ admissible cover $f:C\to D$ is a static polynomial if there exists a smooth point $a_\infty\in C$ such that the local degree of $f$ at $a_\infty$ is $d$. 
\end{Def}

Note that if $\phi$ is PCF and satisfies criteria (\ref{crit:KC}, \ref{item:KC1}) and (\ref{crit:KC}, \ref{item:KC2}), then $\H_{\phi}$ and $\Hbar_{\phi}$ are moduli spaces that parametrize static polynomials. 

\subsection{Degenerations of static polynomials}

Here, we describe a few basic features of the combinatorics of static polynomial admissible covers.  

\begin{lem}\label{lem:connectedpiece}
Suppose $f:C\to D$ is a degree $d$ admissible cover with an irreducible component $C_\infty$ of $C$ such that the restriction $f|_{C_\infty}$ has full degree equal to $d$. If $\theta\in D$ is any node, and $D_0$ is the connected component of $D\setminus\{\theta\}$ that contains $f(C_\infty)\setminus\{\theta\}$, then $f^{-1}(D_0)$ is connected. 
\end{lem}
\begin{proof}
Set $\bar{D_0}$ be the closure of $D_0$ in $D$, and set $\bar{C_0}$ to be the closure of $f^{-1}(D_0)$ in $C$. Then the restriction $f|_{\bar{C_0}}:\bar{C_0}\to \bar{D_0}$ is also an admissible cover of degree $d$, and it has full degree equal to $d$ on the irreducible component $C_\infty$ of $\bar{C_{0}}$. Thus the source curve $\bar{C_0}$ must be connected. Since $\bar{C_0}\setminus f^{-1}(D_0)$ is a set of isolated smooth points (these are nodes of $C$ but smooth points of $\bar{C_0}$), connectedness of $\bar{C_0}$ is equivalent of to connectedness of $ f^{-1}(D_0)$; we conclude that the latter is connected, as desired. 
\end{proof}

\begin{cor}\label{cor:staticpolyconnectedpiece}
Suppose $f:C\to D$ is a degree $d$ static polynomial admissible cover, fully ramified over a smooth point $b_\infty\in D$. If $\theta\in D$ is any node, and $D_0$ is the connected component of $D\setminus\{\theta\}$ that contains $b_\infty$, then $f^{-1}(D_0)$ is connected. 
\end{cor}
\begin{proof}
This follows from Lemma \ref{lem:connectedpiece} and the fact that if $C_\infty$ is the irreducible component containing the fully ramified smooth point $a_\infty=f^{-1}(b_\infty)$, then the restriction $f|_{C_\infty}$ has full degree equal to $d$.
\end{proof}

We further conclude that if $f:C\to D$ is a degenerate static polynomial as in Corollary \ref{cor:staticpolyconnectedpiece}, then the restriction of $f$ to any irreducible component of $C$ is a static polynomial of possibly smaller degree. More precisely, let $C_1$ be an irreducible component of $C$. If $a_\infty\in C_1$, then it's clear that $f|_{C_1}$ is a static polynomial. Otherwise, let $\eta$ be the node on $C_1$ connecting it to $a_\infty$; we claim that $\eta$ is a fully ramified point of $f|_{C_1}$. To see this, set $\theta=f(\eta)$, and $D_0$ to be the connected component of $D\setminus\{\theta\}$ that contains $b_\infty$. Since by Corollary \ref{cor:staticpolyconnectedpiece}, $f^{-1}(D_0)$ is connected, and since $C$ has genus zero, there is a unique node connecting $C_1$ to $f^{-1}(D_0)$; this node must be $\eta$. Thus $\eta$ is the only point of $C_1$ mapping to $\theta$; this forces $f|_{C_1}$ to be fully ramified at $\eta$.



\subsection{Static polynomials and weighted stable curves}

In this section we study Hurwitz spaces $\H$ parametrizing static polynomials. We will relate the combinatorics of static polynomials to the combinatorics of stable curves to find compactifications $X_{\bB}$ and $X_{\bA}$ on which the Hurwitz correspondence induced by $\H$ behaves well. 

\begin{Def}\label{def:CompatiblePair}
Let $\H=\H(\bA,\bB,d,F,\br,\rm)$ be a Hurwitz space parametrizing static polynomials, and let $b_{\infty}\in\bB$ be the image of the fully ramified point, i.e. we have $\br(b)=(d)$. We define a \emph{compatible pair of heavy/light Hassett spaces} with respect to $\H$ to be a pair $X_{\bB}$ and $X_{\bA}$ of compactifications of $\M_{0,\bB}$ and $\M_{0,\bA}$ respectively, obtained as follows. Let $\bB=\bB_{\hvy}\sqcup\bB_{\lt}$ be a set partition such that: $b_\infty\in\bB_{\hvy}$, $\abs{\bB_{\hvy}}\ge2$, and $\abs{F^{-1}(\bB_{\hvy})}\ge2$. Set $\bA_{\hvy}=F^{-1}(\bB_{\hvy})$, and $\bA_{\lt}=F^{-1}(\bB_{\lt})$. Let $\epsilon>0$ be such that $\abs{\bB_{\lt}}<(1/\epsilon)$ and $\abs{\bA_{\lt}}<(1/\epsilon)$. Let $\be_{\bB}$ be the weight datum on $\M_{0,\bB}$ that assigns points in $\bB_{\hvy}$ weight $1$ and points in $\bB_{\lt}$ weight $\epsilon$, and let $\be_{\bA}$ be the weight datum on $\M_{0,\bA}$ that assigns points in $\bA_{\hvy}$ weight $1$ and points in $\bA_{\lt}$ weight $\epsilon$. Set $X_{\bB}$ and $X_{\bA}$ to be the corresponding spaces of $\bB$- and $\bA$-marked weighted stable curves respectively. 

In other words, we require the special point $b_\infty$ (the marked image of the fully ramified point) to be heavy, we require all of the points in $\bA$ that map, under $F$, to heavy points in $\bB$ to be heavy themselves, and we require points in $\bA$ that map to light points in $\bB$ to be light. 
\end{Def}

Now, we fix $\H$, together with a pair of compatible pair $X_{\bB}$ and $X_{\bA}$ of heavy/light Hassett spaces, along with all the notation in Definition \ref{def:CompatiblePair}. Let $\rho_{\bB}$ and $\rho_{\bA}$ be the reduction morphisms from $\Mbar_{0,\bB}$ and $\Mbar_{0,\bA}$ to $X_{\bB}$ and $X_{\bA}$ respectively. We are interested in studying the correspondence induced by $\H$ from $X_{\bB}$ to $X_{\bA}$. In order to be able to use an admissible covers compactification, we pass to the fully marked Hurwitz space: Let $\H^{\full}=\H(\bA^{\full},\bB,d,F,\br,\rm)$ be the fully marked Hurwitz space of $\H$ as in Section \ref{sec:adm}, with $\bA^{\full}\supseteq\bA$ the full marked preimage of $\bB$. Let $\Hbar^{\full}$ be the admissible covers compactification of $\H^\full$; it admits a covering map $\nu$ to $\H$. Set $\bar{\pi_{\bB}}$ and $\bar{\pi_{\bA^{\full}}}$ to the maps from $\Hbar^{\full}$ to $\Mbar_{0,\bB}$ and $\Mbar_{0,\bA^{\full}}$ respectively, and $\mu:\Mbar_{0,\bA^{\full}}\to\Mbar_{0,\bA}$ to be the forgetful map. Throughout the section we will refer back to the notation defined above:
\begin{align}
\begin{array}{ccccccccccccc}\label{eq:AllData}
\bA, & \bA_{\hvy}, & \bA_{\lt}, & \bB, & \bB_{\hvy},& \bB_{\lt}, & \bA^{\full}, & a_\infty, & b_\infty, & d, &F, & \br, & \rm, \\
  \epsilon,&  \be_{\bA}, & \be_{\bB}, & \H, & \Hbar^{\full}, & X_{\bB}, & X_{\bA}, & \bar{\pi_{\bB}}, & \bar{\pi_{\bA^{\full}}},& \rho_{\bB}, & \rho_{\bA}, &\mu& \text{and }\thickspace \nu
\end{array}
\end{align}

\begin{lem}\label{lem:unstabletounstable}
With notation as in (\ref{eq:AllData}), suppose we have $[f:(C,\bA^{\full})\to (D,\bB)]\in\Hbar^\full$. Then, considering $C$ as a (not necessarily stable) $\bA$-marked curve, we have:
\begin{enumerate} 
\item (Statement about the map of curves.) If $C_1$ is an irreducible component of $C$ such that $f(C_1)$ is not $\be_\bB$-stable, then $C_1$ is not $\be_\bA$-stable. \label{it:unstabletounstablecomps}
\item (Equivalent statement about the induced map of dual trees.) If $v$ is a vertex on the dual tree of $C$ such that $f_{\Verts}(v)$ on the dual tree of $D$ is not $\be_\bB$-stable, then $v$ is not $\be_{\bA}$-stable.\label{it:unstabletounstableverts}
\end{enumerate}
\end{lem}
\begin{proof}
Since the equivalence of items (\ref{it:unstabletounstablecomps}) and (\ref{it:unstabletounstableverts}) is clear, we prove only (\ref{it:unstabletounstablecomps}). Since $(D,\bB)$ is a stable curve and $f(C_1)$ is not $\be_\bB$-stable, we conclude from part (\ref{it:HvyLtStableCurves2}) of Lemma \ref{lem:HvyLtStable} that there is a unique node $\theta$ on $f(C_1)$ connecting it to every heavy point, i.e. to every point in $\bB_{\hvy}$. Since $b_\infty$ is heavy, $\theta$ connects $f(C_1)$ to $b_\infty$. Now, let $D_0$ be the connected component of $D\setminus\{\theta\}$ that contains $b_\infty$ (and contains every other point in $\bB_{\hvy}$, and does not contain $f(C_1)\setminus\{\theta\}$). By Corollary \ref{cor:staticpolyconnectedpiece}, $C_0:=f^{-1}(D_0)$ is connected. Since the pair $\be_\bA$ and $\be_\bB$ is a compatible pair of weights as in Definition \ref{def:CompatiblePair}, every point in $\bA_{\hvy}$ maps, via, $f$, to some point in $\bB_{\hvy}$, and thus every point in $\bA_{\hvy}$ lies on $C_0$. Now, since $C$ is genus zero, there is a unique note $\eta$ on $C_1$ connecting it to $C_0$, i.e. $\eta$ is the node on $C_1$ that connects it to every point in $\bA_{\hvy}$. By the criterion in part (\ref{it:HvyLtStableCurves1}) of Lemma \ref{lem:HvyLtStable}, we conclude that $C_1$ is not $\be_\bA$-stable.
\end{proof}

\begin{lem}\label{lem:TwoMapsSameProjection}
With notation as in (\ref{eq:AllData}), suppose $G_{\gamma}$ is any boundary stratum of $\Hbar^\full$ and that $J$ is any irreducible component of $G_{\gamma}$. Then the two maps $(\rho_{\bB}\circ\bar{\pi_{\bB}})$ and $(\rho_{\bA}\circ\mu\circ\bar{\pi_{\bA^\full}})$ from $J$ to $X_{\bB}$ and to $X_{\bA}$ respectively both factor through the projection
\begin{align}\label{eq:AdmStratumProjection}
J=\prod_{w\in\Verts(\tau)}\bar{\mathcal{J}_w}\to\prod_{\substack{w\in\Verts(\tau)\\\be_{\bB}\text{-stable}}}\bar{\mathcal{J}_w},
\end{align}
where the decomposition of $J$ as a product is as per Section \ref{sec:StratificationOfAdmissibleCovers}, equation (\ref{eq:BdryCompAsProd}).
\end{lem}

\begin{proof}
Recall that $\tau$ is the dual tree of the target curve $(D,\bB)$ of a generic admissible cover $[f:(C,\bA^\full)\to (D,\bB)]\in J\subset\Hbar^{\full}$. As described in Section \ref{sec:StratificationOfAdmissibleCovers}, the above decomposition of $J$ into a product is induced by the decomposition $G_\gamma=\prod_{w\in\Verts(\tau)}\Hbar_w,$ where
  $\Hbar_w$ is an admissible covers space of (pure) dimension $\md(w)$. The factor $\bar{\mathcal{J}_w}$ in the decomposition of $J$ is an irreducible component of
  $\Hbar_w$, and thus also has dimension $\md(w)$. Under $\bar{\pi_\bB}$, $J$ maps to the boundary stratum $T_\tau$ in $\Mbar_{0,\bB}$, and the restriction $\bar{\pi_{\bB}}:J\to  T_\tau$ decomposes into a product as follows:

\begin{center}
  \begin{tikzpicture}
    \matrix(m)[matrix of math nodes,row sep=3em,column sep=14em,minimum
    width=2em] {
    J=  \prod_{w\in\Verts(\tau)}\bar{\mathcal{J}}_w&\prod_{w\in\Verts(\tau)}\Mbar_{0,\Flags_w}=T_\tau\\
    };
    \path[-stealth] (m-1-1) edge node [above] {$\bar{\pi_\bB}|_{J}=\prod_{w\in\Verts(\tau)}\bar\pi_w^{\mathrm{target}}$} (m-1-2);
  \end{tikzpicture}
\end{center}
Each factor map $\bar\pi_w^{\mathrm{target}}$ is a finite flat map from the admissible covers space $\bar{\mathcal{J}}_w$ to an appropriate moduli space of target curves. Thus $\bar{\pi_{\bB}}:J\to  T_\tau$ is finite and flat, so $T_\tau$ is the full image of $J$. Now, by Lemma \ref{lem:discardunstable}, the restriction $\rho_\bB:T_\tau \to X_{\bB}$ factors through the projection  
\begin{align}
T_{\tau}=\prod_{w\in\Verts(\tau)}\Mbar_{0,\Flags_w} \to \prod_{\substack{w\in\Verts(\tau)\\\text{$\be_\bB$-stable}}}\Mbar_{0,\Flags_w} \label{eq:TargetStratumProjection}
\end{align}
We conclude that the restriction $\bar{\pi_{\bB}}:J\to X_{\bB}$ factors through the projection in (\ref{eq:AdmStratumProjection}).

Now, the boundary stratum $S_\sigma$ in $\Mbar_{0,\bA^{\full}}$ is isomorphic to
$\prod_{v\in\Verts(\sigma)}\Mbar_{0,\Flags_v},$ and the restriction $\bar{\pi_{\bA^{\full}}}:J\to S_{\sigma}$ also factors as a product:
\begin{center}
  \begin{tikzpicture}
    \matrix(m)[matrix of math nodes,row sep=2em,column sep=15em,minimum
    width=2em] {
     J= \prod\limits_{w\in\Verts(\tau)}\bar{\mathcal{J}}_w&\prod\limits_{w\in\Verts(\tau)}\prod\limits_{v\in
        f_{\Verts}^{-1}(w)}\Mbar_{0,\Flags_v}=S_{\sigma}\\
    }; \path[-stealth] (m-1-1) edge node [above] {$\bar{\pi_\bA^{\full}}|_{J}=\prod_{w\in\Verts(\tau)}\bar\pi_w^{\mathrm{source}}$} (m-1-2);
  \end{tikzpicture}
\end{center}
Each factor map $\bar\pi_w^{\mathrm{source}}$ is a map from the space $\bar{\mathcal{J}}_w$ of admissible covers to a moduli space of possibly disconnected source curves. Note that every vertex on $\sigma$ that is $\be_\bA$-stable is also $\bA$-stable, thus by Lemmas \ref{lem:forget} and \ref{lem:discardunstable}, the restriction $\rho_\bA\circ\mu:S_\sigma \to X_{\bA}$ factors through the composition of projections  
\begin{align}
S_{\sigma}=\prod_{v\in\Verts(\sigma)}\Mbar_{0,\Flags_v} \xrightarrow{}{} \prod_{\substack{v\in\Verts(\sigma)\\\text{$\bA$-stable}}}\Mbar_{0,\Flags_v} \xrightarrow{}{} \prod_{\substack{v\in\Verts(\sigma)\\\text{$\be_\bA$-stable}}}\Mbar_{0,\Flags_v}
\end{align}

From item (\ref{it:unstabletounstableverts}) of Lemma \ref{lem:unstabletounstable}, we conclude that if $v$ is a $\be_{\bA}$-stable vertex on $\sigma$, then $f_{\Verts}(v)$ is $\be_\bB$-stable as a vertex on $\tau$. Thus the restriction $\rho_\bA\circ\mu:S_\sigma \to X_{\bA}$ factors through the projection  

%

\begin{align}
S_{\sigma}=\prod_{v\in\Verts(\sigma)}\Mbar_{0,\Flags_v} \to \prod_{\substack{v\in\Verts(\sigma)\\\text{$f_{\Verts}(v)$ $\be_\bB$-stable}}}\Mbar_{0,\Flags_v}=\prod_{\substack{w\in\Verts(\tau)\\\text{$\be_\bB$-stable}}}\prod\limits_{\thickspace \thickspace v\in
        f_{\Verts}^{-1}(w)}\Mbar_{0,\Flags_v}
\end{align}
Thus the composite $\rho_{\bA}\circ\mu\circ\bar{\pi_\bA^{\full}}:J\to X_{\bA}$ factors though the projection in (\ref{eq:AdmStratumProjection}), proving the lemma.
\end{proof}

Using Lemma \ref{lem:TwoMapsSameProjection}, we conclude that any irreducible component $J$ of the boundary of $\Hbar^\full$ that is contracted in dimension by the map to $X_{\bB}$ must also be contracted in dimension by the map to $X_{\bA}$.

\begin{lem}\label{lem:TwoMapsContractDimension}
With notation in (\ref{eq:AllData}), suppose $G_{\gamma}$ is a boundary stratum of $\Hbar^\full$ and $J$ is some irreducible component of $G_{\gamma}$ such that $\dim_{\C}(\rho_{\bB}\circ\bar{\pi_{\bB}}(J))<\dim_{\C}(J)$. Then $\dim_{\C}(\rho_{\bA}\circ\mu\circ\bar{\pi_{\bA^\full}}(J))<\dim_\C(J)$.
\end{lem}

\begin{proof}
The map $\bar{\pi_{\bB}}$ is finite, so $\dim_\C(\bar{\pi_{\bB}}(J))=\dim_\C(J)$. In fact, $\bar{\pi_{\bB}}(J)$ is the boundary stratum $T_\tau$ of $\Mbar_{0,\bB}$. We conclude that $\dim_\C(\rho_\bB(T_\tau))<\dim_\C(T_\tau)$. By Lemma \ref{lem:TwoMapsSameProjection}, $\rho_{\bA}\circ\mu\circ\bar{\pi_\bA^{\full}}:J\to X_{\bA}$ factors though the projection in \ref{eq:AdmStratumProjection}. We have that $\dim_\C(J)=\sum_{w\in\Verts(\tau)}\md(w)$. By the above, 
$\dim_{\C}(\rho_{\bA}\circ\mu\circ\bar{\pi_{\bA^\full}}(J))\le \sum_{\substack{w\in\Verts(\tau)\\\be_{\bB}\text{-stable}}}\md(w).$ 
By Lemma \ref{lem:CriterionForKernel}, $\tau$ has at least one vertex with positive moduli dimension that is not $\be_{\bB}$-stable. Thus $\sum_{\substack{w\in\Verts(\tau)\\\be_{\bB}\text{-stable}}}\md(w)<\sum_{w\in\Verts(\tau)}\md(w)$,
proving the lemma.
\end{proof}

\begin{rem}\label{rem:posdimfibers}
Since $\bar{\pi_{\bB}}$ is a finite map, any positive-dimensional fiber of $\rho_\bB\circ\bar{\pi_{\bB}}$ is the pre-image, under $\bar{\pi_{\bB}}$, of a positive-dimensional fiber of $\rho_{\bB}$. In turn, any positive-dimensional fiber of $\rho_{\bB}$ is a union of fibers of projections from boundary strata in $\Mbar_{0,\bB}$ onto factors corresponding to $\be_{\bB}$-stable vertices of their dual trees. Thus, the pre-image, under $\bar{\pi_{\bB}}$, of such a positive-dimensional fiber of $\rho_{\bB}$ is a union of fibers of projections from boundary strata in $\Hbar$ onto factors corresponding to $\be_{\bB}$-stable vertices of their target dual trees. By Lemma \ref{lem:TwoMapsContractDimension}, $\rho_\bA\circ\mu\circ\bar{\pi_{\bA^{\full}}}$ contracts to a point each such fiber. We conclude that any connected component of some positive-dimensional fiber of $\rho_\bB\circ\bar{\pi_{\bB}}$ maps to a single point under $\rho_\bA\circ\mu\circ\bar{\pi_{\bA^{\full}}}$. As a consequence, the image of $\Hbar^\full$ in the product $X_{\bB}\times X_{\bA}$ is finite over $X_{\bB}$. This implies that the correspondence induced by $\Hbar$ from $X_{\bB}$ to $X_{\bA}$ is regular.
\end{rem}

\begin{prop}\label{prop:KerneltoKernel}
With notation in (\ref{eq:AllData}), let $\Gamma$ be any non-empty union of connected components of $\H$. Then, for $k=0,\ldots,\dim_{\C}(\Mbar_{0,\bB})$, $[\Gamma]_*:H_{2k}(\Mbar_{0,\bB})\to H_{2k}(\Mbar_{0,\bA})$ takes $\ker((\rho_{\bB})_*)$ to $\ker((\rho_{\bA})_*)$.
\end{prop}

\begin{rem}
The very beginning of the proof of Proposition \ref{prop:KerneltoKernel} follows the beginning the proof of Theorem 9.7 of \cite{Ramadas2015}: In order to understand the pushforward by a Hurwitz correspondence on the homology groups of the Deligne-Mumford compactifications we reduce to the case of a fully marked Hurwitz space, then frame the action of the pushforward on boundary strata in terms of the stratification of the space of admissible covers. After this point, the two proofs diverge. Here, the key content lies in Lemma \ref{lem:TwoMapsSameProjection}, via Lemma \ref{lem:TwoMapsContractDimension}.
\end{rem}

\begin{proof}
First, we reformulate the problem in order to allow ourselves to work solely with fully marked Hurwitz spaces. Set $\Gamma^{\full}=\nu^{-1}(\Gamma)$ and set $\bar{\Gamma^{\full}}$ to be its closure in $\Hbar^{\full}$. Fix $k\in\{0,\ldots,\dim_{\C}(\Mbar_{0,\bB})\}$. By Lemma 7.2 of \cite{Ramadas2015}, we have $[\Gamma]_*=(1/\deg{\nu}) [\Gamma^{\full}]_*$. Thus it suffices to show that $[\Gamma^\full]_*=\mu_*\circ((\bar{\pi_\bA^{\full}})|_{\bar{\Gamma^\full}})_*\circ(\bar{\pi_\bB})|_{\bar{\Gamma^\full}}^{*}$ sends $\ker((\rho_{\bB})_*)$ to $\ker((\rho_{\bA})_*)$. Suppose $[T_\tau]\in\ker((\rho_{\bB})_*)$ is an arbitrary $k$-dimensional boundary stratum in the kernel of $(\rho_{\bB})_*$. By Lemma \ref{lem:CriterionForKernel}, $\tau$ has at least one vertex with positive moduli dimension that is not $\be_{\bB}$-stable. Also, $\dim_\C(\rho_{\bB}(T_\tau))<\dim_\C(T_{\tau})$. Since the map $\bar{\pi_{\bB}}$ is flat, by Lemma 1.7.1 of \cite{Fulton1998} we have that
 $$(\bar{\pi_{\bB}})|_{\bar{\Gamma^\full}}^*([T_{\tau}])=\sum_{\substack{J \text{ irred. comp}\\\text{of }(\bar{\pi_{\bB}})|_{\bar{\Gamma^\full}}^{-1}(T_{\tau})}} m_J [J],$$
where the $(m_J)$s are positive integer multiplicities. Let $J$ be an arbitrary term appearing in the above sum; since $\bar{\pi_{\bB}}$ is finite, $J$ is an irreducible component of some $k$-dimensional boundary stratum $G_{\gamma}$ of $\Hbar^{\full}$. We have that 
$$\dim_\C(\rho_{\bB}\circ\bar{\pi_{\bB}}(J))=\dim_\C(\rho_{\bB}(T_\tau))<\dim_\C(T_{\tau})=\dim_\C(J).$$
 By Lemma \ref{lem:TwoMapsContractDimension}, $\dim(\rho_{\bA}\circ\mu\circ(\bar{\pi_\bA^{\full}})|_{\bar{\Gamma^\full}}(J) < \dim(J),$ 
so $(\rho_{\bA}\circ\mu\circ(\bar{\pi_\bA^{\full}})|_{\bar{\Gamma^\full}})_*([J])=0\in H_{2k}(X_{\bA})$. Since $J$ was arbitrary, we conclude:
\begin{align*}
(\rho_{\bA})_*([\Gamma^\full]_*([T_\tau]))=&(\rho_{\bA})_*\circ\mu_*\circ((\bar{\pi_\bA^{\full}})|_{\bar{\Gamma^\full}})_*\circ\bar{\pi_\bB}|_{\bar{\Gamma^\full}}^{*}([T_{\tau}])\\
=&\sum_{\substack{J \text{ irred. comp}\\\text{of }(\bar{\pi_{\bB}})|_{\bar{\Gamma^\full}}^{-1}(T_{\tau})}} m_J (\rho_{\bA}\circ\mu\circ(\bar{\pi_\bA^{\full}})|_{\bar{\Gamma^\full}})_*[J]=0\in H_{2k}(X_{\bA}).
\end{align*}

 Thus $[\Gamma^\full]_*([T_\tau])\in\ker((\rho_{\bA})_*)$. We conclude that for an arbitrary boundary stratum $T_{\tau}$ of $\Mbar_{0,\bB}$ in the kernel of $(\rho_{\bB})_*$, its pushforward $[\Gamma^\full]_*([T_{\tau}])$ is in the kernel of $(\rho_{\bA})_*$. By \cite{Ramadas2015}, Lemma 10.9, $\ker((\rho_{\bB})_*)$ is generated by the fundamental classes of boundary strata. We conclude that $[\Gamma^\full]_*$ sends $\ker((\rho_{\bB})_*)$ to $\ker((\rho_{\bA})_*)$, proving the proposition.
\end{proof}

\section{Algebraic stability of $\H_{\phi}$ on heavy/light Hassett spaces}\label{sec:Proof}

In this section, we fix a degree $d$ branched covering $\phi$ with finite post-critical set $\bP$ that satisfies criteria (\ref{crit:KC}, \ref{item:KC1}) and (\ref{crit:KC}, \ref{item:KC2}). Let $p_\infty\in\bP$ be the fully ramified cyclic point, and let $\bP_{\infty}\subset\bP$ be the forward orbit of $p_\infty$ i.e. the set of all points in its periodic cycle. Fix a positive rational number $\epsilon$ such that $\frac{1}{\abs{\bP}-1}<\epsilon<\frac{1}{\abs{\bP}-2}$ and let $\be$ be the weight datum that assigns weight $1$ to the elements of $\bP_{\infty}$ and assigns weight $\epsilon$ to all elements not in $\bP_{\infty}$. Let $X=\Mbar_{0,\bP}(\be)$ be the corresponding Hassett space. If $\abs{\bP_{\infty}}=1$ then $p_\infty$ is fixed so $\phi$ is a topological polynomial. Also, as described in Section \ref{sec:tower}, the Hassett space $X$ is isomorphic to $\CP^{\abs{\bP}-3}$. By \cite{Koch2013}, $R_{\phi}$ is holomorphic thus algebraically stable on $X$. If $\abs{\bP_{\infty}}\ge 2$ then $X$ is a heavy/light Hassett space. Here, we show:

\begin{thm}\label{thm:algebraicstability}
Let $\phi,\bP,p_\infty, \bP_{\infty}, \epsilon, \be$, and $X$ be as above, and suppose $\abs{\bP_{\infty}}\ge2$. Then $\H_{\phi}$ and $R_{\phi}$ are algebraically stable on $X$.
\end{thm}

\begin{proof}
By \cite{Ramadas2015}, The Hurwitz correspondence $\H_{\phi}$ is algebraically stable on $\Mbar_{0,\bP}$. Now, let $\rho:\Mbar_{0,\bP}\to X$ be the reduction map. Note that $(X,X)$ is a compatible pair of heavy/light Hassett spaces with respect to the Hurwitz space $\H_{\phi}$, as in Definition \ref{def:CompatiblePair}. Thus $\H_{\phi}$, $X$ and $\rho$ together satisfy the assumptions of Proposition \ref{prop:KerneltoKernel}. We conclude that for all $k$, the kernel of $\rho_*:H_{2k}(\Mbar_{0,\bP})\to H_{2k}(X)$ is an invariant subspace of $[\H_{\phi}]_*$. On the other hand, it is shown in \cite{Ramadas2015} Lemma 4.16 that if a correspondence $\Gamma$ is algebraically stable on $X_1$, and if $r:X_1\to X_2$ is a regular birational map such that for all $m$, the kernel of the pushforward $r_*:H_{m}(X_1)\to H_{m}(X_2)$ is invariant under the action of $[\Gamma]_*$ on $H_m(X_1)$, then $\Gamma$ is also algebraically stable on $X_2$. Applying this result here tells us that $\H_{\phi}$ is algebraically stable on $X$, i.e. for every iterate $n$, $[\H_{\phi}^n]_*=[\H_\phi]_*^n$ on $H_{2k}(X)$. On the other hand, for all $k=0,\ldots,\dim_\C(X)$, and for iterates $n>0$, the action of $(R_{\phi}^n)^*$ on $H^{2k}(X)=H^{k,k}(X)$ is identified with the action of $[\H_{\phi}^n]_*$ on $H_{2(\dim_\C(X)-k)}(X)$. We conclude that $R_{\phi}$ is algebraically stable on $X$. 
\end{proof}

Theorem \ref{thm:Main} follows as an immediate consequence: it is a restatement of Theorem \ref{thm:algebraicstability} above.

\begin{rem}
There is a variant of Theorem \ref{thm:algebraicstability} obtained by applying Proposition \ref{prop:KerneltoKernel} repeatedly. Suppose $\phi$ is a degree $d$ branched covering with finite post-critical set $\bP$ that satisfies criterion (\ref{crit:KC}, \ref{item:KC1}) with $p_\infty$ the fully ramified periodic point, and such that every critical point of $\phi$ is periodic. Then $\phi$ satifies (\ref{crit:KC}, \ref{item:KC2}) as well. Let $\bP_{\infty}$ and $\epsilon$ be as in the statement of Theorem \ref{thm:algebraicstability}. Let $\bP_{\infty}= \bP_1\subset\bP_2\ldots,\bP_r=\bP$ be any filtration of $\bP$ such that each $\bP_i$ is a union of periodic cycles of $\phi$. For $i=1,\ldots, r$, let $\be_i$ be the weight datum on $\M_{0,\bP}$ assigning weight $1$ to points in $\bP_i$ and weight $\epsilon$ to points in the complement of $\bP_i$, and let $X_i$ be the corresponding space of weighted stable curves. Note that there is a generalized reduction map $\rho_{i,i-1}:X_i\to X_{i-1}$ commuting with the reduction maps $\rho_i$ and $\rho_{i-1}$ from $X_r=\Mbar_{0,\bP}$ to $X_i$ and $X_{i-1}$ respectively. For $i=1,\ldots, r-1$, set $V_i,k=\ker((\rho_i)_*)\subset H_{2k}(\Mbar_{0,\bP})$. By Proposition \ref{prop:KerneltoKernel}:
\begin{enumerate}
\item The subspace $V_{i,k}$ is invariant under the action $[\H_{\phi}]_*=R_{\phi}^*$ on $H_{2k}(\Mbar_{0,\bP})=H^{2(\dim(\Mbar_{0,\bP})-k)}(\Mbar_{0,\bP})$. Thus by Lemma 4.16 of \cite{Ramadas2015}:
\item The Hurwitz correspondence $H_{\phi}$ and the rational map $R_{\phi}$ are algebraically stable on each $X_i$. 
\end{enumerate}
Thus $V_{r-1,k}\subset V_{r-2,k}\subset\ldots\subset H_{2k}(\Mbar_{0,\bP})=H^{2(\dim(\Mbar_{0,\bP})-k)}(\Mbar_{0,\bP})$ is a filtration of $H^{2(\dim(\Mbar_{0,\bP})-k)}(\Mbar_{0,\bP})$ by $R_{\phi}^*$-invariant subspaces. This filtration allows us to write $R_{\phi}^*$ as a block-lower-triangular matrix. In contrast, the main result of \cite{Ramadas2015} is a completely different filtration of the (co)homology groups of $\Mbar_{0,\bP}$ by subspaces invariant for every Hurwitz correspondence, giving us in this specific context another --- different and utterly independent --- expression of $R_{\phi}^*$ as a block-lower-triangular matrix. The latter expression makes no use of the specifics of $\phi$, in particular of criteria (\ref{crit:KC}, \ref{item:KC1}) and (\ref{crit:KC}, \ref{item:KC2}).

\end{rem}
\bibliographystyle{amsalpha}
\bibliography{../../HurwitzRefs}

\newcommand{\etalchar}[1]{$^{#1}$}
\providecommand{\bysame}{\leavevmode\hbox to3em{\hrulefill}\thinspace}
\providecommand{\MR}{\relax\ifhmode\unskip\space\fi MR }
\providecommand{\MRhref}[2]{%
  \href{http://www.ams.org/mathscinet-getitem?mr=#1}{#2}
}
\providecommand{\href}[2]{#2}
\begin{thebibliography}{CHMR16}

\bibitem[BC16]{BlancCantat2016}
J\'er\'emy Blanc and Serge Cantat, \emph{Dynamics of birational transformations
  of projective surfaces}, Journal of the American Mathematical Society
  \textbf{29} (2016), 415--471.

\bibitem[BDJ19]{BellDillerJonsson2019}
Jason~P. Bell, Jeffrey Diller, and Mattias Jonsson, \emph{A transcendental
  dynamical degree}, arXiv preprint arXiv:1907.00675 (2019).

\bibitem[Cap]{Caporaso2014}
Lucia Caporaso, \emph{Gonality of algebraic curves and graphs}, Algebraic and
  Complex Geometry. Springer Proceedings in Mathematics and Statistics
  \textbf{71}.

\bibitem[Cey09]{Ceyhan2009}
{\"O}zg{\"u}r Ceyhan, \emph{Chow groups of the moduli spaces of weighted
  pointed stable curves of genus zero}, Advances in Mathematics \textbf{221}
  (2009), 1964--1978.

\bibitem[CHMR16]{CHMRHeavyLight}
Renzo Cavalieri, Simon Hampe, Hannah Markwig, and Dhruv Ranganathan,
  \emph{{Moduli spaces of rational weighted stable curves and tropical
  geometry}}, Forum of Mathematics, Sigma \textbf{4} (2016), E9.

\bibitem[CMR16]{CavalieriMarkwigRanganathan2016}
Renzo Cavalieri, Hannah Markwig, and Dhruv Ranganathan, \emph{Tropicalizing the
  space of admissible covers}, Mathematische Annalen \textbf{364} (2016),
  no.~3, 1275--1313.

\bibitem[DF01]{DillerFavre2001}
Jeffrey Diller and Charles Favre, \emph{Dynamics of bimeromorphic maps of
  surfaces}, American Journal of Mathematics \textbf{123} (2001), no.~6,
  1135--1169.

\bibitem[DH93]{DouadyHubbard1993}
Adrien Douady and John~H. Hubbard, \emph{A proof of {T}hurston's topological
  characterization of rational functions}, Acta Mathematica \textbf{171}
  (1993), no.~2, 263--297.

\bibitem[DS05]{DinhSibony2005}
Tien-Cuong Dinh and Nessim Sibony, \emph{Une borne sup\'erieure pour l'entropie
  topologique d'une application rationelle}, Annals of Mathematics \textbf{161}
  (2005), no.~3, 1637--1644.

\bibitem[DS08]{DinhSibony2008}
\bysame, \emph{Upper bound for the topological entropy of a meromorphic
  correspondence}, Israel Journal of Mathematics \textbf{163} (2008), no.~1,
  29--44.

\bibitem[F{\etalchar{+}}03]{Favre2003}
Charles Favre et~al., \emph{Les applications monomiales en deux dimensions},
  The Michigan Mathematical Journal \textbf{51} (2003), no.~3, 467--475.

\bibitem[Ful98]{Fulton1998}
William Fulton, \emph{Intersection theory}, second ed., Springer-Verlag New
  York, 1998.

\bibitem[Gro03]{Gromov2003}
Mikha{\"i}l Gromov, \emph{On the entropy of holomorphic maps}, Enseign. Math.
  \textbf{49} (2003), 217--235.

\bibitem[Has03]{Hassett2003}
Brendan Hassett, \emph{Moduli spaces of weighted pointed stable curves},
  Advances in Mathematics \textbf{173} (2003), no.~2, 316--352.

\bibitem[HM82]{HarrisMumford1982}
Joe Harris and David Mumford, \emph{On the {K}odaira dimension of the moduli
  space of curves}, Inventiones {M}athematicae \textbf{67} (1982), 23--86.

\bibitem[JW11]{JonssonWulcan2011}
Mattias Jonsson and Elizabeth Wulcan, \emph{Stabilization of monomial maps},
  Michigan Mathematical Journal \textbf{60} (2011), no.~3, 629--660.

\bibitem[Kee92]{Keel1992}
Sean Keel, \emph{Intersection theory on the moduli space of stable $n$-pointed
  curves of genus zero}, Transactions of the American Mathematical Society
  \textbf{330} (1992), no.~2.

\bibitem[KKL19]{KannanKarpLi2019}
Siddarth Kannan, Dagan Karp, and Shiyue Li, \emph{Chow rings of heavy/light
  hassett spaces via tropical geometry}, Preprint; arXiv:1910.10883 (2019).

\bibitem[Koc13]{Koch2013}
Sarah Koch, \emph{Teichm{\"u}ller theory and critically finite endomorphisms},
  Advances in Mathematics \textbf{248} (2013), 573--617.

\bibitem[KR16]{KochRoeder2015}
Sarah Koch and Roland K.~W. Roeder, \emph{Computing dynamical degrees of
  rational maps on moduli space}, Ergodic Theory and Dynamical Systems
  \textbf{36} (2016), no.~8, 2538--2579.

\bibitem[Lin12a]{LinStability}
Jan-Li Lin, \emph{Algebraic stability and degree growth of monomial maps},
  Mathematische Zeitschrift \textbf{271} (2012), 293–311.

\bibitem[Lin12b]{Lin2012}
\bysame, \emph{Pulling back cohomology classes and dynamical degrees of
  monomial maps}, Bull. Soc. Math. France \textbf{140} (2012), no.~4, 533--549
  (eng).

\bibitem[Ram17]{RamadasThesis}
Rohini Ramadas, \emph{Dynamics on the moduli space of pointed rational curves},
  Ph.D. thesis, University of Michigan, 2017.

\bibitem[Ram18]{Ramadas2015}
\bysame, \emph{{Hurwitz correspondences on compactifications of
  $\overline{\mathcal{M}}_{0,N}$}}, Advances in Mathematics \textbf{323}
  (2018), 622--667.

\bibitem[Ram19]{Ramadas2016}
Rohini Ramadas, \emph{Dynamical degrees of hurwitz correspondences}, Ergodic
  Theory and Dynamical Systems (2019), 1–23.

\bibitem[Roe13]{Roeder2013}
Roland~K.W. Roeder, \emph{The action on cohomology by compositions of rational
  maps}, Mathematical Research Letters \textbf{22} (2013), no.~2.

\bibitem[RW06]{RomagnyWewers2006}
Matthieu Romagny and Stefan Wewers, \emph{{H}urwitz spaces}, S{\'e}min. Congr.
  \textbf{13} (2006), 313--341.

\bibitem[Smy09]{Smyth2009}
David~Ishii Smyth, \emph{Towards a classification of modular compactifications
  of ${\cM}_{g,n}$}, Inventiones {M}athematicae \textbf{192} (2009), no.~2,
  459--503.

\bibitem[{Tru}15]{Truong2015}
Tuyen~Trung {Truong}, \emph{{(Relative) dynamical degrees of rational maps over
  an algebraic closed field}}, ArXiv e-prints (2015),
  \href{http://arxiv.org/abs/1501.01523}{\texttt{arXiv:1501.01523}}.

\bibitem[TT16]{Truong2016}
Tuyen Trung~Truong, \emph{Relative dynamical degrees of correspondences over a
  field of arbitrary characteristic}, Journal für die reine und angewandte
  Mathematik (Crelles Journal) (2016).

\bibitem[Yom87]{Yomdin1987}
Yosef Yomdin, \emph{Volume growth and entropy}, Israel Journal of Mathematics
  \textbf{57} (1987), no.~3, 285--300.

\end{thebibliography}
\end{document}